\documentclass[a4paper]{article}
\usepackage[authoryear,round]{natbib}
\usepackage{main}
\usepackage[left=1.in, right=1.in, bottom=1.5in, top=1.5in]{geometry}
\setcitestyle{authoryear,round,citesep={;},aysep={,},yysep={;}}
\let\cite\citep

\title{Beyond Polytopes: Damped Newton Frank--Wolfe Methods over Compact Convex Sets}
\author{Shota Takahashi\thanks{
Graduate School of Information Science and Technology, The University of Tokyo, Tokyo, Japan (\href{mailto:shota@mist.i.u-tokyo.ac.jp}{\texttt{shota@mist.i.u-tokyo.ac.jp}})}
\and
Akiko Takeda\thanks{
Graduate School of Information Science and Technology, The University of Tokyo, Tokyo, Japan (\href{mailto:takeda@mist.i.u-tokyo.ac.jp}{\texttt{takeda@mist.i.u-tokyo.ac.jp}})}~\thanks{Center for Advanced Intelligence Project, RIKEN, Tokyo, Japan}}
\date{\today}

\begin{document}

\maketitle

\begin{abstract}
    We study convex optimization over compact convex sets for twice differentiable objective functions with positive definite Hessians, assuming access to the feasible set through a linear minimization oracle.
    Existing second-order Frank--Wolfe methods either provide only a local linear rate over general convex sets or achieve global linear and local quadratic convergence only over polytopes.
    We propose \emph{damped Newton Frank--Wolfe methods} that approximately solve constrained damped Newton subproblems using Frank--Wolfe or away-step Frank--Wolfe inner iterations.
    We develop residual backtracking, using a root Newton stepsize as a safe reference, together with a switching rule that eventually takes full steps.
    Under H\"older smoothness assumptions for $p$th derivatives, we establish global linear convergence and local convergence with Q-order at least $1+\nu$ ($p=2$) and local quadratic convergence ($p=3$).
    Experiments on matrix sensing and ridge-regularized logistic regression show that the residual-backtracking variant is consistently robust and often achieves the best performance among the tested first- and second-order baselines.
\end{abstract}

\section{Introduction}
In this paper, we focus on the constrained convex optimization problem of the form
\begin{align}
    \min_{x \in \X} \quad f(x) \label{prob:constrained}
\end{align}
where $f:\R^n \to \R$ is a twice differentiable convex function whose Hessian $\nabla^2 f(x)$ is positive definite for all $x \in \X$, and $\X \subset \R^n$ is a compact convex set.
Let $x^\star$ denote a minimizer of~\eqref{prob:constrained}.
We assume access to a linear minimization oracle (LMO) for $\X$, \ie, given $c\in\R^n$, we can compute $v \in \argmin_{u\in\X}\langle c,u\rangle$.
The \emph{Frank--Wolfe method} (FW) is a projection-free method that requires only LMO calls to solve~\eqref{prob:constrained} rather than projection oracles.
A projection oracle can be computationally expensive for complex constraint sets, whereas an LMO is often more efficient as shown in \citep{Combettes2021-ct,Braun2025-sz,Woodstock2026-mz}.
From this fact, the FW method and LMO have been widely used in machine learning optimizers~\citep{Sfyraki2026-mv} and signal processing~\citep{Odor2016-sn,Takahashi2026-ch} as well as in optimization~\citep{Braun2025-sz}.

\paragraph{Related Work.}
The FW method was originally proposed by~\citet{Frank1956-cq} and was independently developed and extended as the \emph{conditional gradient method} (CG) by~\citet{Levitin1966-sm}.
The first lower bounds on the convergence rate of the FW method were established by~\citet{Canon1968-sr} and later refined by~\citet{GueLat1986-bg}.
Subsequent work sharpened the convergence analysis and clarified the limitations of FW methods.
In particular,~\citet{Jaggi2013-oz} refined the convergence theory and derived a lower bound relating iterate sparsity to solution accuracy.
A closely related lower bound was obtained by~\citet{Lan2013-dv} through an oracle-complexity analysis of methods based on linear optimization.
Faster convergence rates can be obtained under additional geometric assumptions on the feasible set. 
Both a faster sublinear rate and a linear convergence rate over strongly convex sets were derived by~\citet{Garber2015-de}, and the analysis was extended to uniformly convex sets by~\citet{Kerdreux2021-kc}.
On polytopes, however, the classical FW method can zigzag near an optimal face, leading to slow convergence.
The \emph{away-step FW method} was proposed by~\citet{Wolfe1970-ax} to mitigate zigzagging, and its linear convergence was established by~\citet{Lacoste-Julien2015-gb}.
\citet{Pedregosa2020-ay} proposed a backtracking line-search strategy, and~\citet{Takahashi2026-ch} extended this strategy to the Bregman setting.
For a comprehensive review of FW methods, we refer the interested reader to the survey by~\citet{Braun2025-sz}.

\paragraph{Limitations of Existing Methods.}

Even for strongly convex objective functions, the classical FW method can exhibit only sublinear convergence without additional geometric assumptions on $\X$. Thus, unlike projected gradient methods, it does not generally admit a linear convergence guarantee.

The \emph{conditional gradient sliding method} (CGS) was proposed by~\citet{Lan2016-op} to reduce the number of gradient evaluations for smooth convex problems while retaining the optimal LMO complexity.
It incorporates Nesterov's acceleration and approximately solves the projection subproblems using multiple LMO calls, allowing each gradient evaluation to be reused across several inner iterations (see~\cref{appendix:additional-related-work} for acceleration and projection methods).
For strongly convex objectives, CGS achieves a linear rate in terms of gradient evaluations without additional geometric assumptions on $\X$.
More recently, line-search-free adaptive stepsizes were incorporated into CGS by~\citet{Takahashi2026-tp}, removing the need for prior knowledge of a global smoothness constant.

Second-order FW methods have also been developed to obtain faster convergence.
The \emph{Newton FW method} (NFW), proposed by~\citet{Liu2022-gs}, uses second-order information to achieve local linear convergence for constrained self-concordant problems (see~\cref{appendix:additional-related-work} for second-order methods).
NFW is globally convergent, but its linear rate is guaranteed only in the full-step phase.
The \emph{second-order conditional gradient sliding method} (SOCGS), proposed by~\citet{Carderera2026-pk}, approximately solves constrained Newton subproblems using an away-step FW method and achieves global linear convergence and local quadratic convergence for strongly convex objectives over polytopes.
These guarantees are restricted to polytopes, with strict complementarity required for the local quadratic rate.
Setting the subproblem accuracy also requires a lower bound on the primal gap $f(x_k) - f(x^\star)$ at the $k$th iteration.
Then the following question naturally arises:

\begin{quote}
    \emph{Can a projection-free Newton method combine global linear convergence with local quadratic convergence over compact convex sets?}
\end{quote}

\paragraph{Contributions.}

We answer the above question in the affirmative by developing new Newton FW methods for general compact convex sets.
Our main contributions are summarized as follows.

\begin{itemize}
    \item \textbf{Algorithmic framework.}
    We propose a damped Newton FW method (\cref{alg:dnfw}) and a variant with an away-step FW inner loop (\cref{alg:dnfw_away}).
    Unlike existing second-order FW methods~\citep{Liu2022-gs,Carderera2026-pk}, our methods are based on damped constrained Newton subproblems.
    \item \textbf{Stepsize and switching strategies.}
    We develop residual backtracking for the damped Newton FW method.
    Since the underlying Newton subproblems are constrained, their design and analysis differ from those for unconstrained Newton methods.
    We also introduce a switching rule that eventually sets $\alpha_k=1$, thereby enabling local quadratic convergence.
    \item \textbf{Convergence guarantees.}
    We establish global linear convergence over compact convex sets without additional geometric assumptions.
    Switching yields local convergence with Q-order at least $1+\nu$ when $p=2$ and $\nu > 0$, and local quadratic convergence when $p=3$.
    Moreover, the inner accuracy does not require a lower bound on the primal gap.
    \item \textbf{Numerical performance.}
    Our method is consistently robust and often outperforms the tested first- and second-order baselines.
\end{itemize}

\cref{tab:convergence_rates} compares the convergence rates of the proposed method with those of existing FW methods.
\cref{tab:total_lmo_complexity} in \cref{appendix:total_lmo_complexity} shows the comparison of the total LMO complexity.

\begin{table}[!tbp]
    \caption{Global and local complexities (measured in outer iterations) of LMO-based methods for achieving $f(x_k)-f(x^\star)\leq\epsilon$. r-CGS denotes restarted CGS. SOCGS requires strict complementarity for its local quadratic rate. Our $\C^2$ bounds require $\nu>0$. Total LMO complexities are reported in \cref{tab:total_lmo_complexity} (\cref{appendix:total_lmo_complexity}).}
    \label{tab:convergence_rates}
    \centering
    \begin{tabular}{cccccll}
        \toprule
        Algorithm & $f$ & $\nabla f$ Lip. & $\nabla^p f$ & $\X$ & global & local \\ \midrule
        FW & $\C^1$ cvx & \cmarkg & \xmarkr & cvx & $\O(\epsilon^{-1})$ & ---\\
        CGS & $\C^1$ cvx & \cmarkg & \xmarkr & cvx & $\O(\epsilon^{-1/2})$ & ---\\
        r-CGS & $\C^1$ scvx & \cmarkg & \xmarkr & cvx & $\O(\log\epsilon^{-1})$ & ---\\
        NFW & $\C^3$ cvx & \xmarkr & self-conc. & cvx & --- & $\O(\log \epsilon^{-1})$\\
        SOCGS & $\C^2$ scvx & \cmarkg & $\nabla^2 f$ Lip. & polytope & $\O(\log \epsilon^{-1})$ & $\O(\log\log \epsilon^{-1})$\\
        \textbf{Alg.~\ref{alg:dnfw},~\ref{alg:dnfw_away}} & $\C^2$ scvx & \cmarkg & $\nabla^2 f$ H\"older & cvx & $\O(\log \epsilon^{-1})$ & $\O(\log\log \epsilon^{-1})$ \\
        \textbf{Alg.~\ref{alg:dnfw},~\ref{alg:dnfw_away}} & $\C^3$ scvx & \cmarkg & $\nabla^3 f$ H\"older & cvx & $\O(\log \epsilon^{-1})$ & $\O(\log\log\epsilon^{-1})$ \\
        \bottomrule
    \end{tabular}
\end{table}

\paragraph{Notation.}
We use $\|\cdot\|$ to denote the Euclidean norm and $\langle\cdot,\cdot\rangle$ to denote the Euclidean inner product for vectors and Frobenius inner product for matrices.
For a positive definite matrix $H$, define $\|y\|_H\coloneq\langle Hy,y\rangle^{1/2}$ and $\|h\|_H^*\coloneq\langle h,H^{-1}h\rangle^{1/2}$.
We abbreviate $\|y\|_x\coloneq\|y\|_{\nabla^2f(x)}$ and $\|h\|_x^*\coloneq\|h\|_{\nabla^2f(x)}^*$, and define $D\coloneq\sup_{x,y\in\X}\|x-y\|_x$.
Given $x\in\X$ (usually we set $x = x_k$), the operator norm of a matrix $H\in\R^{n\times n}$ induced by the Euclidean norm, and the operator norms of $H$ and a third-order tensor $T\in\R^{n\times n\times n}$ are defined as
\begin{align*}
    \|H\|_2\coloneq\sup_{y\neq0}\frac{\|Hy\|}{\|y\|}, \qquad \|H\|_{\op} \coloneq \sup_{y\neq0}\frac{\|Hy\|_x^*}{\|y\|_x}, \qquad \|T\|_{\op} \coloneq \sup_{y,z,w\neq0}\frac{|T[y,z,w]|}{\|y\|_x\|z\|_x\|w\|_x},
\end{align*}
respectively.
We define the FW gap at $x_k$ by $g_k\coloneq\max_{v\in\X}\langle\nabla f(x_k),x_k-v\rangle$.
For a set $P\subset\R^n$, $\conv P$ and $\vertex P$ denote its convex hull and set of extreme points, respectively.

\section{Proposed Method: Damped Newton Frank--Wolfe Method}\label{sec:proposed-method}

We propose projection-free methods for solving~\eqref{prob:constrained} based on second-order information of $f$.

\paragraph{Damped Newton FW Method.}

We define the damped quadratic model of $f$ at $x_k$ as
\begin{align*}
    \Psi_\alpha(x;x_k)\coloneq f(x_k)+\langle\nabla f(x_k),x-x_k\rangle+\frac{1}{2\alpha}\|x-x_k\|_{x_k}^2.
\end{align*}
When $\alpha=1$, $\Psi_\alpha(\cdot;x_k)$ coincides with the second-order Taylor model of $f$ at $x_k$.
For $\alpha\in(0,1)$, the coefficient of the quadratic term is increased, and hence deviations from $x_k$ are penalized more strongly.
Thus, $\alpha$ serves as a stepsize, or equivalently, as a damping parameter for the Newton step.
Indeed, minimizing $\Psi_\alpha(\cdot;x_k)$ over $\X$ is equivalent to computing the scaled projection
\begin{align}
    \argmin_{x\in\X}\Psi_\alpha(x;x_k)=\argmin_{x\in\X}\left\|x-\left(x_k-\alpha\nabla^2f(x_k)^{-1}\nabla f(x_k)\right)\right\|_{x_k}^2. \label{eq:quadratic-subproblem}
\end{align}
Consequently, for $\X=\R^n$, the unique minimizer is $x_k-\alpha\nabla^2f(x_k)^{-1}\nabla f(x_k)$.
For constrained problems, minimizing $\Psi_\alpha(\cdot;x_k)$ over $\X$ yields a projected variable-metric step~\citep{Carderera2026-pk}; projected Newton methods under simple constraints were studied in~\citep{Bertsekas1982-hl}.
This fits the inexact proximal Newton method via an indicator function~\citep{Lee2014-zv}; our methods achieve global linear and fast local convergence over compact convex sets (see~\cref{sec:convergence-analysis}).

\begin{wrapfigure}[16]{R}{0.5\textwidth}
\vspace{-1.0\baselineskip}
\begin{algorithm}[H]
    \caption{Damped Newton FW Method}\label{alg:dnfw}
    \KwIn{\mbox{$x_0\in\X$, $\alpha_k\in(0,1]$, $\eta_k\geq0$}}
    \For{$k=0,1,2,\dots$}{
        \makebox[0pt][l]{$x_{k+1} \gets \fw(\nabla f(x_k), \frac{1}{\alpha_k}\nabla^2 f(x_k), x_k, \eta_k)$}\;
    }
\end{algorithm}
\begin{algorithm}[H]
    \caption{$\texttt{FW}(h, H, u, \eta)$}\label{alg:fw}
    $u_0 \gets u$\;
    \For{$t=0,1,2,\dots$}{
        $v_t \gets \argmin_{v \in \X} \langle h + H(u_t - u), v \rangle$\;
        \If{$\langle h + H(u_t - u), u_t - v_t \rangle \leq \eta$}{
            \Return $u_t$\;
        }
        $\gamma_t \gets \min\left\{1, \frac{\langle h + H(u_t - u), u_t - v_t \rangle}{\|v_t - u_t\|_H^2}\right\}$\;
        $u_{t+1} \gets (1 - \gamma_t) u_t + \gamma_t v_t$\;
    }
\end{algorithm}
\end{wrapfigure}

Minimizing $\Psi_\alpha(\cdot;x_k)$ requires solving a constrained quadratic problem, which may be computationally expensive.
We therefore approximately minimize $\Psi_\alpha(\cdot;x_k)$ using FW, which accesses $\X$ only through LMO calls.
Newton FW methods based on the undamped $\Psi_1(\cdot;x_k)$ have been proposed independently~\citep{Liu2022-gs,Carderera2026-pk}.
\citet{Liu2022-gs} provide only a local linear rate, whereas the convergence guarantees of \citet{Carderera2026-pk} are restricted to polytopes.

In this paper, we consider the damped quadratic model $\Psi_\alpha(\cdot;x_k)$ with $\alpha\in(0,1]$.
We now present the \emph{damped Newton FW method} (\cref{alg:dnfw}), which approximately minimizes $\Psi_\alpha(\cdot;x_k)$ over $\X$ using $\fw$ (\cref{alg:fw}).
We terminate \cref{alg:fw} when the FW gap of $\Psi_\alpha(\cdot;x_k)$ is sufficiently small, \ie, when it holds that
\begin{align}
    G_t\coloneq\max_{v\in\X}\langle\nabla\Psi_{\alpha_k}(u_t;x_k),u_t-v\rangle\leq \eta_k. \label{ineq:fw_termination}
\end{align}
We distinguish the model FW gap $G_t$ at the $t$th inner iteration from the FW gap $g_k$ of $f$ at the $k$th iteration.
\cref{alg:fw} is efficient because it uses the exact stepsize $\gamma_t = \argmin_{\gamma\in[0,1]}\Psi_{\alpha_k}((1-\gamma)u_t + \gamma v_t; x_k)$ with $h = \nabla f(x_k)$, $H = \frac{1}{\alpha_k}\nabla^2 f(x_k)$, and $u = x_k$.
The choices of $\alpha_k$ and $\eta_k$ in \cref{sec:convergence-analysis} guarantee convergence of the outer iterates.
We analyze the complexity of the inner FW iterations in \cref{alg:fw} up to an accuracy $\eta_k$. The proof is given in \cref{appendix:proof-lmo_complexity}.
\begin{proposition}
    \label[proposition]{proposition:lmo_complexity}
    Let $\alpha_k \in (0,1]$ and $\eta_k > 0$ be given.
    For any $k\geq0$, the number of inner iterations at the $k$th outer iteration is at most $\lceil24D^2/(\alpha_k\eta_k)\rceil$.
\end{proposition}

\paragraph{Away-Step FW Inner Loop.}
We may also use away-step FW for~\eqref{eq:quadratic-subproblem}.
For this variant, we assume that the LMO returns a vertex of $\X$, including at initialization.
It maintains a finite active set $\Scal_t\subset\vertex\X$ with $u_t\in\conv\Scal_t$ and supplements standard FW steps toward a vertex with away steps that can remove unfavorable active vertices.
On polytopes, these steps can mitigate zigzagging near the optimal face~\citep{Wolfe1970-ax,GueLat1986-bg}.
For strongly convex $L$-smooth objectives over polytopes, it converges linearly~\citep{Lacoste-Julien2015-gb}.
At each outer iteration, SOCGS~\citep{Carderera2026-pk} computes an away-step FW update on $f$ and an approximate solution of~\eqref{eq:quadratic-subproblem} with $\alpha_k\equiv1$, then selects the point with lower objective value.
Our method uses away-step FW only to approximately solve the subproblem, avoiding the extra update on $f$ and candidate comparison while preserving the outer convergence analysis.

We present the damped Newton FW method with an away-step FW inner loop in \cref{alg:dnfw_away} (see~\cref{alg:afw} in~\cref{appendix:proof-lmo_complexity_away} for $\afw$).

\begin{algorithm}[!htbp]
    \caption{Damped Newton FW Method with $\protect\afw$}\label{alg:dnfw_away}
    \KwIn{$x_0\in\argmin_{v\in\X}\langle \nabla f(x), v\rangle$ for $x \in \X$, $\alpha_k\in(0,1]$, $\eta_k\geq0$
    }
    $\widehat{\Scal}_0 \gets \{x_0\}$, $\widehat\lambda_0 \gets 1$\;

    \For{$k=0,1,2,\dots$}{
        $(x_{k+1}, \widehat{\Scal}_{k+1}, \widehat\lambda_{k+1})\gets\afw(\nabla f(x_k), \frac{1}{\alpha_k}\nabla^2f(x_k), x_k, \eta_k, \widehat{\Scal}_k, \widehat\lambda_k)$\;
    }
\end{algorithm}

When $\X$ is a polytope, we analyze the complexity of the inner iterations up to $\eta_k$.
The dependence on the inner accuracy is logarithmic rather than inverse-linear.
The proof is given in \cref{appendix:proof-lmo_complexity_away}.

\begin{proposition}
    \label[proposition]{proposition:lmo_complexity_away}
    Let $\alpha_k \in (0,1]$ and $\eta_k > 0$ be given and let $\X$ be a polytope that is not a singleton.
    For any $k\geq0$, the number of inner iterations at the $k$th outer iteration is at most $|\widehat{\Scal}_k|-1+2\left\lceil\frac{\log\max\{1,g_k/\zeta_k\}}{-\log\left(1-\mu_k\delta_{\X}^2/(4L_kd_{\X}^2)\right)}\right\rceil$, where $\delta_{\X}$ is the pyramidal width of $\X$, $d_{\X}\coloneq\max_{x,y\in\X}\|x-y\|$, $\mu_k\coloneq\lambda_{\min}(\nabla^2f(x_k))/\alpha_k$, $L_k\coloneq\lambda_{\max}(\nabla^2f(x_k))/\alpha_k$, and $\zeta_k\coloneq\min\left\{L_kd_{\X}^2/2,\eta_k^2/(2L_kd_{\X}^2)\right\}$.
\end{proposition}

\section{Convergence Analysis}\label{sec:convergence-analysis}
We now establish the global convergence and the local convergence of \cref{alg:dnfw,alg:dnfw_away}.

\subsection{Global Convergence}\label{subsec:global-convergence}

We first state H\"older continuity assumptions of the $p$th derivative of $f$ used in the analysis.
\begin{assumption}\label[assumption]{assumption:holder}
    The objective function $f$ is $p$ times differentiable and its $p$th derivative is H\"older continuous, \ie, for some $\nu \in [0,1]$, there exists a constant $L_{p,\nu} > 0$ such that, for all $x,y\in\X$,
    \begin{align}
        \|\nabla^p f(x) - \nabla^p f(y)\|_{\op} \leq L_{p,\nu} \|x-y\|_{x}^\nu.
    \end{align}
    When $p=3$, suppose that $\Gamma_{3,\X}\coloneq\sup_{x\in\X}\|\nabla^3f(x)\|_{\op}$ is finite. Throughout the $p=3$ analysis, $L_{3,\nu}$ denotes a constant chosen sufficiently large so that $\Gamma_{3,\X}\leq2\left(L_{3,\nu}/(1+\nu)\right)^{1/(1+\nu)}$.
\end{assumption}

Let $p\in\{2,3\}$, $q\coloneq p+\nu$, $\rho\in(0,1)$, and $\omega>0$.
We define $r_k\coloneq\nabla f(x_k)+s_{k-1}$ and $w_k\coloneq\nabla f(x_k)+s_k$, where $s_k\coloneq-\nabla f(x_k) - \frac{1}{\alpha_k}\nabla^2 f(x_k)(x_{k+1}-x_k)$.
We initialize $s_{-1}=0$, $\eta_{-1}=0$, and $\Delta_{-1}=0$, and set $\Delta_k\coloneq\max\{\|r_k\|_{x_k}^*,\rho\Delta_{k-1}\}$.
If the FW gap $g_k=0$, the method terminates at an optimal solution; otherwise, $\Delta_k>0$.
The inner accuracy for global convergence is
\begin{align}
    \eta_k\coloneq\frac{\omega\Delta_k^2}{1+B_q\Delta_k^{(q-2)/(q-1)}}, \label{eq:global_inner_accuracy}
\end{align}
where $B_q\coloneq(9L_{p,\nu})^{\frac{1}{q-1}}\kappa_\omega^{\frac{q-2}{q-1}}$ and $\kappa_\omega\coloneq\frac{1+\sqrt{1+4\omega(1+\rho^{-2})}}{2}$.

\paragraph{Root Newton Stepsize.}
\citet{Hanzely2026-kv} proposed a root Newton stepsize for unconstrained problems.
We consider $\alpha_k\coloneq\frac{1}{1+\theta_k}$ and $\theta_k\coloneq B_q\Delta_k^{(q-2)/(q-1)}$.
As shown below, this value guarantees acceptance and therefore provides a safe endpoint for backtracking.
We instead select $\theta_k$ by backtracking and use $\alpha_k=(1+\theta_k)^{-1}$, while keeping \eqref{eq:global_inner_accuracy} fixed throughout the search.

\paragraph{Residual Backtracking.}
\citet{Hanzely2026-kv} also proposed the universal backtracking strategy for unconstrained problems.
We extend it to constrained problems with inexact updates and cap the trial damping parameter.
We present the \emph{residual backtracking strategy} in \cref{alg:residual_backtracking}.
Starting with $\bar\theta_0>0$ and $\tau>1$, it uses the procedure in \cref{alg:residual_backtracking} at each outer iteration.
The trial quantities $x_{k,j}$, $s_{k,j}$, $w_{k,j}$, and $r_{k+1,j}$ are computed as in the algorithm, and the first trial satisfying
\begin{align}
    \langle r_{k+1,j},\nabla^2f(x_k)^{-1}w_{k,j}\rangle\geq\frac{c_p}{\alpha_{k,j}\theta_{k,j}}\|r_{k+1,j}\|_{x_k}^{*2} \label{ineq:backtracking}
\end{align}
is accepted.
Here, $c_p\coloneq1/2$ for $p=2$ and $c_p\coloneq1/4$ for $p=3$.
The accepted $\theta_k$ need not equal the root value, and $\eta_k$ is not recomputed from it.

\begin{algorithm}[htbp]
    \caption{Residual Backtracking for Damped Newton FW}
    \label{alg:residual_backtracking}
    \KwIn{$x_k$, $\Delta_k>0$, $\bar\theta_k>0$, $\tau>1$}
    Compute $\eta_k$ using \eqref{eq:global_inner_accuracy}\;
    \For{$j=0,1,2,\dots$}{
        $\theta_{k,j}\gets\min\{\tau^j\bar\theta_k,B_q\Delta_k^{(q-2)/(q-1)}\}$, $\alpha_{k,j}\gets1/(1+\theta_{k,j})$\;
        $x_{k,j}\gets\fw(\nabla f(x_k),\alpha_{k,j}^{-1}\nabla^2f(x_k),x_k,\eta_k)$\label{line:inner_fw}\;
        $s_{k,j}\gets-\nabla f(x_k)-\alpha_{k,j}^{-1}\nabla^2f(x_k)(x_{k,j}-x_k)$\;
        $w_{k,j}\gets\nabla f(x_k)+s_{k,j}$, $r_{k+1,j}\gets\nabla f(x_{k,j})+s_{k,j}$\;
        \If{$\langle r_{k+1,j},\nabla^2f(x_k)^{-1}w_{k,j}\rangle\geq\frac{c_p}{\alpha_{k,j}\theta_{k,j}}\|r_{k+1,j}\|_{x_k}^{*2}$}{
            $\theta_k\gets\theta_{k,j}$, $\alpha_k\gets\alpha_{k,j}$, $x_{k+1}\gets x_{k,j}$, $s_k\gets s_{k,j}$, $\bar\theta_{k+1}\gets\theta_k/\tau$\;
            \textbf{break}\;
        }
    }
\end{algorithm}

Each call to $\fw$ may be replaced by $\afw$ in Line~\ref{line:inner_fw}, initialized with a copy of the current active-set representation.
Only the active set and weights of the accepted trial are carried to the next outer iteration.
The following residual estimate applies to both strategies and to each trial satisfying the inner stopping condition.
Its proof is given in \cref{appendix:proof-q_r_relation_inexact}.
\begin{lemma}\label[lemma]{lemma:q_r_relation_inexact}
    For every $k\geq0$, the following inequality holds:
    \begin{align}
        \|w_k\|_{x_k}^{*2}\leq\|r_k\|_{x_k}^*\|w_k\|_{x_k}^*+\frac{\eta_k+\eta_{k-1}}{\alpha_k}. \label{ineq:q_r_inexact_basic}
    \end{align}
    With \eqref{eq:global_inner_accuracy} at each iteration, $\alpha_k\geq(1+B_q\Delta_k^{(q-2)/(q-1)})^{-1}$ implies $\|w_k\|_{x_k}^*\leq\kappa_\omega\Delta_k$.
\end{lemma}

Residual backtracking satisfies the following one-step decrease; its proof is given in~\cref{appendix:proof-inexact_one_step}.
\begin{theorem}[Sufficient decrease]\label{theorem:inexact_one_step}
    Suppose that \cref{assumption:holder} holds and use \eqref{eq:global_inner_accuracy}.
    For $\Delta_k>0$, let $x_{k+1}\in\X$ satisfy \eqref{ineq:fw_termination}.
    If either $\theta_k=B_q\Delta_k^{(q-2)/(q-1)}$ or $\theta_k$ is accepted by \cref{alg:residual_backtracking}, then
    \begin{align}
        f(x_k)-f(x_{k+1})\geq\frac{c_p}{B_q}\frac{\|r_{k+1}\|_{x_k}^{*2}}{\Delta_k^{(q-2)/(q-1)}}-\eta_k. \label{ineq:one_step_decrease_inexact}
    \end{align}
\end{theorem}

For $w_{k,j}\neq0$, \cref{proposition:residual_acceptance} in \cref{appendix:residual_acceptance} gives the following sufficient condition for \eqref{ineq:backtracking}:
\begin{align}
    \theta_{k,j}\geq(9L_{p,\nu})^{\frac{1}{q-1}}\|w_{k,j}\|_{x_k}^{*\frac{q-2}{q-1}}. \label{ineq:backtracking_sufficient}
\end{align}
The root value guarantees acceptance by \cref{lemma:q_r_relation_inexact}; the case $w_{k,j}=0$ is immediate.
Consequently, truncation guarantees finite backtracking, as stated below and proved in \cref{appendix:proof-residual_backtracking}.

\begin{corollary}[Finite backtracking]\label[corollary]{corollary:residual_backtracking}
    Suppose that \cref{assumption:holder} holds with $p\in\{2,3\}$ and $\nu\in[0,1]$.
    Let $\bar\theta_0>0$ and $\tau>1$.
    With the accuracy \eqref{eq:global_inner_accuracy}, \cref{alg:residual_backtracking} terminates its inner solves and backtracking after finitely many iterations at every nonterminal outer iteration.
\end{corollary}

\paragraph{Global Linear Convergence.}
Under \cref{assumption:holder} with $q>2$, there exist $\mu,L>0$ such that $\mu I\preceq\nabla^2f(x)\preceq LI$ for all $x\in\X$ (see \cref{lemma:uniform_hessian_bounds}).
We use the sufficient decrease just established to control the potential, without requiring residual backtracking to accept the root stepsize.
The proof of the following proposition is given in \cref{appendix:proof-potential_decrease}.

\begin{proposition}\label[proposition]{proposition:potential_decrease}
    Suppose that \cref{assumption:holder} holds with $q>2$.
    Let $\omega>0$ satisfy
    \begin{align}
        0<\omega<\frac{c_p\mu}{L}\rho^2(1-\rho^{q/(q-1)}). \label{ineq:omega_condition}
    \end{align}
    Then, there exist constants $\chi,c_\chi>0$ such that $\Phi(x_k)\coloneq f(x_k)-f(x^\star)+\chi\Delta_k^{\frac{q}{q-1}}$ satisfies
    \begin{align}
        \Phi(x_k)-\Phi(x_{k+1})\geq c_\chi\frac{\Delta_{k+1}^2}{\Delta_k^{\frac{q-2}{q-1}}}. \label{ineq:potential_decrease}
    \end{align}
\end{proposition}

Condition \eqref{ineq:omega_condition} ensures that $\chi>0$ can be chosen in \eqref{ineq:chi_interval}, with $c_\chi$ defined by \eqref{eq:c_chi}.
We now state the global convergence of~\cref{alg:dnfw,alg:dnfw_away}; its proof is given in \cref{appendix:proof-global_linear} (see also \cref{appendix:proofs-global-convergence} for terminology).

\begin{theorem}[Global R-linear convergence]\label{theorem:global_linear}
    Suppose that \cref{assumption:holder} holds with $q>2$ and that $\omega$ satisfies \eqref{ineq:omega_condition}.
    Then, for all $k\geq0$,
    \begin{align}
        f(x_k)-f(x^\star) \leq \Phi(x_k) \leq \left(1+\frac{c_\chi\rho^{(q-2)/(q-1)}}{\chi+\left(L/(2\mu)+\omega\rho^{-2}\right)\left(\Phi(x_0)/\chi\right)^{(q-2)/q}}\right)^{-k}\Phi(x_0). \label{ineq:global_linear_obj}
    \end{align}
\end{theorem}

\subsection{Local Convergence}\label{subsec:local-convergence}

We now establish local convergence of \cref{alg:dnfw,alg:dnfw_away} with Q-order at least $1+\nu$ for $p=2$ and Q-quadratic convergence for $p=3$.
To combine global convergence with fast local convergence, we use the FW-gap-based condition below to switch to full steps after finitely many iterations, rather than merely relying on $\alpha_k\to1$.
Upon switching, we stop~\cref{alg:residual_backtracking}, set $\alpha_k=1$, and adopt the following inner stopping rule.
When the FW gap $g_k>0$, which implies $\|\nabla f(x_k)\|>0$, define $M_t\coloneq f(x_k)-\Psi_1(u_t;x_k)$.
For any $a>0$, define the stopping iteration by
\begin{align}
    \ell_k(a)\coloneq\min\left\{t\geq0\ \middle|\ G_t\leq\left(\frac{M_t}{\|\nabla f(x_k)\|}\right)^a\right\}. \label{eq:model_decrease_stopping}
\end{align}
This yields local convergence with Q-order at least $1+\nu$ for both the iterates and the primal gap when $p=2$, and local Q-quadratic convergence of both quantities when $p=3$.
In both cases, the resulting local iteration complexity is $\O(\log\log\epsilon^{-1})$.
More precisely, when $p=2$, the local iteration complexity is $\O(\log\log\epsilon^{-1}/\log(1+\nu))$, which differs only by a constant factor for fixed $\nu > 0$.
The following lemma establishes finite termination of the inner loop and bounds the error of its output.
Its proof is given in \cref{appendix:proof-finite_inner_iterations_error_estimate}.

\begin{lemma}[Finite termination and error estimate]\label[lemma]{lemma:finite_inner_iterations_error_estimate}
    Suppose that $g_k>0$ and $\alpha_k=1$.
    Then, for any $a>0$, the stopping iteration $\ell_k(a)$ in \eqref{eq:model_decrease_stopping} is finite.
    Moreover, the following inequalities hold:
    \begin{align}
        G_{\ell_k(a)}&\leq\|\widehat x_{k+1}-x_k\|^a, \label{ineq:model_decrease_gap_bound}\\
        \|u_{\ell_k(a)}-\widehat x_{k+1}\|&\leq\sqrt{\frac{2}{\mu}}\|\widehat x_{k+1}-x_k\|^{a/2}, \label{ineq:model_decrease_inner_error}
    \end{align}
    where $\widehat x_{k+1}$ is the exact minimizer of $\Psi_1(\cdot;x_k)$ over $\X$.
\end{lemma}

For $p=2$, let $K_2$ be the switching iteration defined in \eqref{eq:p2_model_decrease_switch}.
For $k<K_2$, we use~\cref{alg:residual_backtracking} with the accuracy \eqref{eq:global_inner_accuracy}.
For $k\geq K_2$, if $g_k=0$, the algorithm terminates; otherwise, we set $\alpha_k=1$, run the inner method until \eqref{eq:model_decrease_stopping} with $a=2(1+\nu)$ is satisfied, and set $x_{k+1}\coloneq u_{\ell_k(2(1+\nu))}$.
We now establish the local convergence of \cref{alg:dnfw,alg:dnfw_away} when $p=2$ and $\nu\in(0,1]$. Its proof is given in \cref{appendix:proof-local_convergence_p2}.

\begin{theorem}[Local convergence with Q-order at least $1+\nu$]\label{theorem:local_convergence_p2}
    Suppose that \cref{assumption:holder} holds with $p=2$ and $\nu\in(0,1]$ and that $\omega$ satisfies~\eqref{ineq:omega_condition}.
    Let the switching iteration $K_2$ be defined as
    \begin{align}
        K_2\coloneq\min\left\{k\geq0\mathrel{}\middle|\mathrel{}g_k\leq\frac{\mu}{2}(2C_2)^{-\frac{2}{\nu}}\right\}, \label{eq:p2_model_decrease_switch}
    \end{align}
    where $C_2\coloneq\frac{L_{2,\nu}L^{(1+\nu)/2}}{(1+\nu)\sqrt{\mu}}+\sqrt{\frac{2}{\mu}}\left(1+\frac{L_{2,\nu}L^{(1+\nu)/2}D^\nu}{(1+\nu)\mu^{(1+\nu)/2}}\right)^{1+\nu}$.
    Then, $K_2$ is finite, and either the sequence reaches $x^\star$ in finitely many iterations, or
    \begin{align}
        \|x_{k+1}-x^\star\|\leq C_2\|x_k-x^\star\|^{1+\nu} \label{ineq:p2_model_decrease_local_order}
    \end{align}
    holds for every $k\geq K_2$.
    Moreover, in the latter case, there exists a constant $\widetilde C_2>0$ such that
    \begin{align}
        f(x_{k+1})-f(x^\star)\leq\widetilde C_2\left(f(x_k)-f(x^\star)\right)^{1+\nu} \label{ineq:p2_model_decrease_local_objective}
    \end{align}
    holds for all $k\geq K_2$, yielding convergence of both the iterates and primal gap with Q-order at least $1+\nu$ (Q-quadratic when $\nu=1$).
\end{theorem}

For $p=3$, \cref{assumption:holder} and compactness of $\X$ imply that, for some $M>0$,
\begin{align}
    \|\nabla^2f(x)-\nabla^2f(y)\|_2\leq M\|x-y\|,\qquad x,y\in\X. \label{ineq:p3_hessian_lipschitz}
\end{align}
Let $K_3$ be the switching iteration defined in \eqref{eq:p3_quadratic_switch}.
For $k<K_3$, we likewise use~\cref{alg:residual_backtracking} with \eqref{eq:global_inner_accuracy}.
For $k\geq K_3$, if $g_k=0$, the algorithm terminates; otherwise, we set $\alpha_k=1$, run the inner method until \eqref{eq:model_decrease_stopping} with $a=4$ is satisfied, and set $x_{k+1}\coloneq u_{\ell_k(4)}$.
We now establish the local convergence of \cref{alg:dnfw,alg:dnfw_away} when $p=3$ and $\nu\in[0,1]$. Its proof is given in \cref{appendix:proof-local_convergence_p3}.

\begin{theorem}[Local quadratic convergence]\label{theorem:local_convergence_p3}
    Suppose that \cref{assumption:holder} holds with $p=3$ and $\nu\in[0,1]$ and that $\omega$ satisfies~\eqref{ineq:omega_condition}.
    Let the switching iteration $K_3$ be defined as
    \begin{align}
        K_3\coloneq\min\left\{k\geq0\mathrel{}\middle|\mathrel{}g_k\leq\frac{\mu}{8}C_3^{-2}\right\}, \label{eq:p3_quadratic_switch}
    \end{align}
    where $C_3\coloneq\frac{M}{2\mu}+\sqrt{\frac{2}{\mu}}\left(1+\frac{MD}{2\mu^{3/2}}\right)^2$.
    Then, $K_3$ is finite, and either the sequence reaches $x^\star$ in finitely many iterations, or
    \begin{align}
        \|x_{k+1}-x^\star\|\leq C_3\|x_k-x^\star\|^2 \label{ineq:p3_local_quadratic}
    \end{align}
    holds for every $k\geq K_3$.
    Moreover, in the latter case, there exists a constant $\widetilde C_3>0$ such that
    \begin{align}
        f(x_{k+1})-f(x^\star)\leq\widetilde C_3\left(f(x_k)-f(x^\star)\right)^2 \label{ineq:p3_local_quadratic_objective}
    \end{align}
    holds for all $k\geq K_3$, yielding Q-quadratic convergence of both the iterates and primal gap.
\end{theorem}

\section{Numerical Experiments}\label{sec:numerical-experiments}
In this section, we present numerical experiments to demonstrate the performance of our proposed methods.
All experiments were conducted in Julia 1.12 on a Mac Studio equipped with an Apple M4 Max processor and 128~GB of LPDDR5 memory.
We compared \cref{alg:dnfw,alg:dnfw_away} using residual backtracking (RBNFW; see~\cref{alg:residual_backtracking}) with the Newton FW method (NFW)~\citep{Liu2022-gs}, SOCGS~\citep{Carderera2026-pk}, the FW method with line search, the away-step FW method, CGS~\citep{Lan2016-op}, the projected gradient method (PG), and the accelerated projected gradient method (APG)~\citep{Nesterov1983-xn,Nesterov2018-gj,Beck2009-kr}.
We applied SOCGS and the away-step FW method only to polytope settings.
The descriptions in parentheses following RBNFW specify the definition of $\eta_k$.
The setting labeled \texttt{global} used the accuracy \eqref{eq:global_inner_accuracy} with $p=2$ for RBNFW.
The maximal number of inner iterations was set to 1000 for RBNFW, NFW, CGS, and SOCGS.
When an inner solve reached the cap before satisfying its prescribed accuracy condition, we accepted the last inner iterate as the next outer iterate (see \cref{tab:inner_iterations_reach} in \cref{appendix:inner_iterations_reach} for the number of outer iterations reaching this limit and its discussion).
For the settings labeled \texttt{local2} and \texttt{local3}, we computed the constants in the switching criteria \eqref{eq:p2_model_decrease_switch} and \eqref{eq:p3_quadratic_switch}, respectively, and used \eqref{eq:model_decrease_stopping} after switching, with $a=2(1+\nu)$ and $a=4$, respectively.
We set $\rho=0.625$, $\bar\theta_0 = 0.25$, $\tau = 2$, and $\omega$ to $0.99$ times the right-hand side of \eqref{ineq:omega_condition}.
In all experiments reported here, we set $\nu=1$.
All methods were set to terminate when the FW gap was below $10^{-8}$.
RBNFW, NFW, and SOCGS were limited to 50 outer updates, whereas CGS, FW, away-step FW, PG, and APG were limited to 1000 updates.
Each figure reports the primal gap versus wall-clock time (seconds), the number of iterations, and the number of LMO calls.
Because the optimal objective values for ridge-regularized logistic regression were unknown, we selected the best-performing methods, ran them for a sufficiently large number of iterations, and used the smallest objective value obtained.

\subsection{Matrix Sensing with Squared Loss}\label{subsec:squared-setting}
Let $\S^n$ ($\S_+^n$) denote real symmetric (positive semidefinite) $n\times n$ matrices.
Let $\mathcal A:\S^n\to\R^m$ be the linear operator defined by $\mathcal A(X) = (\langle A_1, X\rangle,\ldots,\langle A_m, X\rangle)^{\T}$, where $A_1,\ldots,A_m\in\R^{n\times n}$ are given sensing matrices.
We consider the matrix sensing problem with squared loss defined by $\min_{X\in\M_{R_{\tr}}}\frac{1}{2}\|\mathcal A(X)-y_{\mathrm{obs}}\|^2$, where $\M_{R_{\tr}} = \{X\in\S_+^n\mid\tr X = R_{\tr}\}$ and $y_{\mathrm{obs}}\in\R^m$ is the vector of observations.
We used a controlled synthetic construction that allowed us to independently prescribe the condition number of the quadratic objective and the alignment of the initial error with the high-curvature subspace.
This enabled us to examine the sensitivity of first- and second-order methods to ill-conditioning.
The details of the construction are given in \cref{appendix:matrix_sensing}.

We set $n = 750$, $m = n(n+1)/2$, $R_{\tr} = 1$, $\varphi = 16$, $\rank X_\star = 1$, and condition number $\Lambda = 1000$ and generated data using seed 2026.
The initial point was $X_0 = (R_{\tr}/n) I$.
We omitted SOCGS because the convergence guarantees in \citet{Carderera2026-pk} do not cover $\M_{R_{\tr}}$.
\cref{fig:matrix_sensing_align0p8} shows the primal gap for $\varpi=0.8$ (see \cref{appendix:matrix_sensing} for $\varpi$).
Results for $\varpi \in \{0.2, 0.5\}$ appear in \cref{appendix:matrix_sensing}.
RBNFW (\texttt{local2}, \texttt{local3}) met their switching criteria by $k=2$ and switched to full steps with the local inner criterion.
RBNFW variants performed best in wall-clock time.

\subsection{Ridge-Regularized Logistic Regression}\label{subsec:ridge-logistic}
Let $A\in\mathbb{R}^{m\times n}$ be a binary-classification design matrix, let $a_i^{\T}$ denote its $i$th row, and let $y_i\in\{-1,+1\}$ be the corresponding label.
We solved the constrained ridge-regularized logistic regression problem $\min_{x\in\X}\frac{1}{m}\sum_{i=1}^{m}\log(1+\exp(-y_i a_i^{\T} x))+\frac{\beta}{2}\|x\|^2$.
We consider two choices of the feasible set $\X$.
The first is the $\ell_2$-ball $\{x\in\R^n\mid\|x\|\leq R\}$ with radius $R>0$.
The second is the sparse polytope $\{x\in\R^n\mid\|x\|_1\leq \xi R_\infty, \|x\|_\infty\leq R_\infty\}$ with $R_\infty>0$ and $\xi\in\{1,\ldots,n\}$.
We used \cref{alg:dnfw} for the $\ell_2$-ball constraint and \cref{alg:dnfw_away} for the sparse polytope constraint.

We set $\beta=10^{-3}$, $R = 1$, $\xi = 10$, and $R_\infty = 1/\sqrt{10}$, and initialized at $x_0 = 0$ for the $\ell_2$-ball and at a vertex $x_0\in\argmin_{x\in\X}\langle\nabla f(0), x\rangle$ for the sparse polytope.
We used the a9a ($m=32561$, $n=123$), covtype.binary ($m=581012$, $n=54$), mushrooms ($m=8124$, $n=112$), phishing ($m=11055$, $n=68$), and w7a ($m=24692$, $n=300$)\footnote{The phishing and w7a datasets were used for sparse polytope constraints.} datasets from the LIBSVM repository~\citep{Chang2011-zh}.
\cref{fig:logistic_covtype_l2_ball,fig:logistic_covtype_polytope} show the primal gap for the covtype.binary dataset with $\ell_2$-ball and sparse polytope constraints, respectively.
\cref{fig:logistic_a9a_polytope} shows the primal gap for the a9a dataset with a sparse polytope constraint.
The RBNFW variants performed best in \cref{fig:logistic_covtype_l2_ball}.
For the sparse polytope constraint, the wall-clock performance of the RBNFW variants was nearly identical to that of SOCGS in \cref{fig:logistic_covtype_polytope}.
Results for the remaining datasets and constraints are reported in \cref{appendix:ridge_logistic_regression}.
Across these datasets, RBNFW consistently exhibited stable and strong performance.

\begin{figure}[p]
    \centering
    \includegraphics[width=\textwidth]{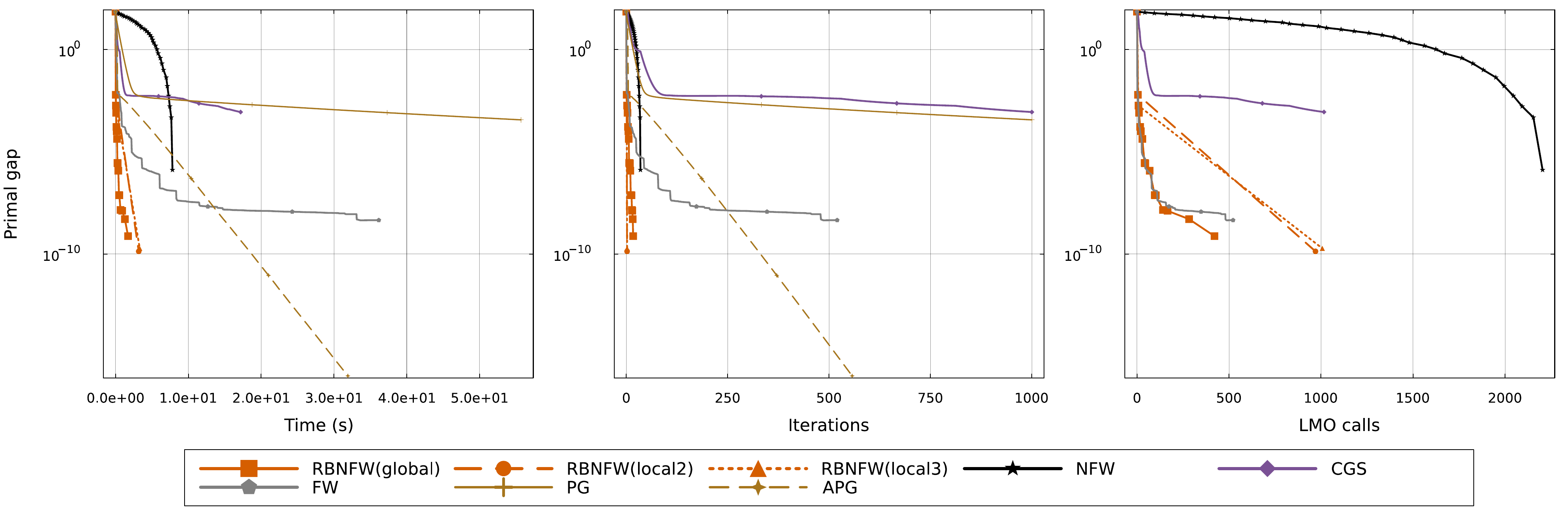}
    \caption{Primal gap for matrix sensing with $\varpi=0.8$ using \cref{alg:dnfw}.}
    \label{fig:matrix_sensing_align0p8}

    \smallskip
    \includegraphics[width=\textwidth]{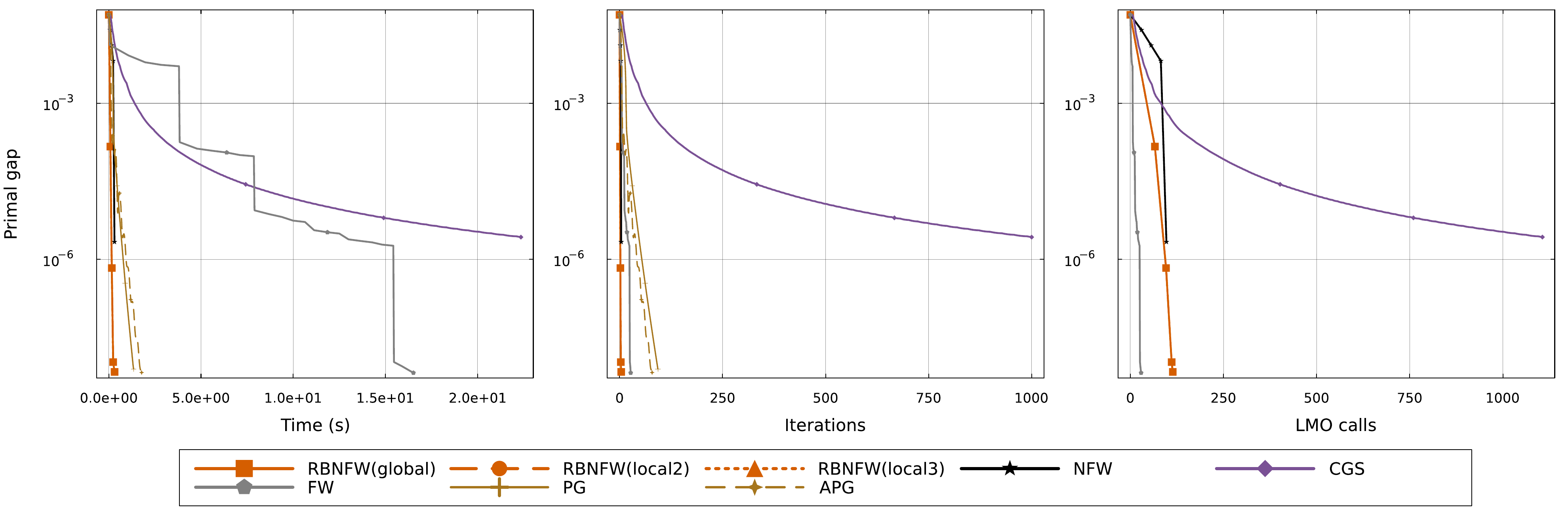}
    \caption{Primal gap for the covtype.binary dataset with an $\ell_2$-ball constraint using \cref{alg:dnfw}.}
    \label{fig:logistic_covtype_l2_ball}

    \smallskip
    \includegraphics[width=\textwidth]{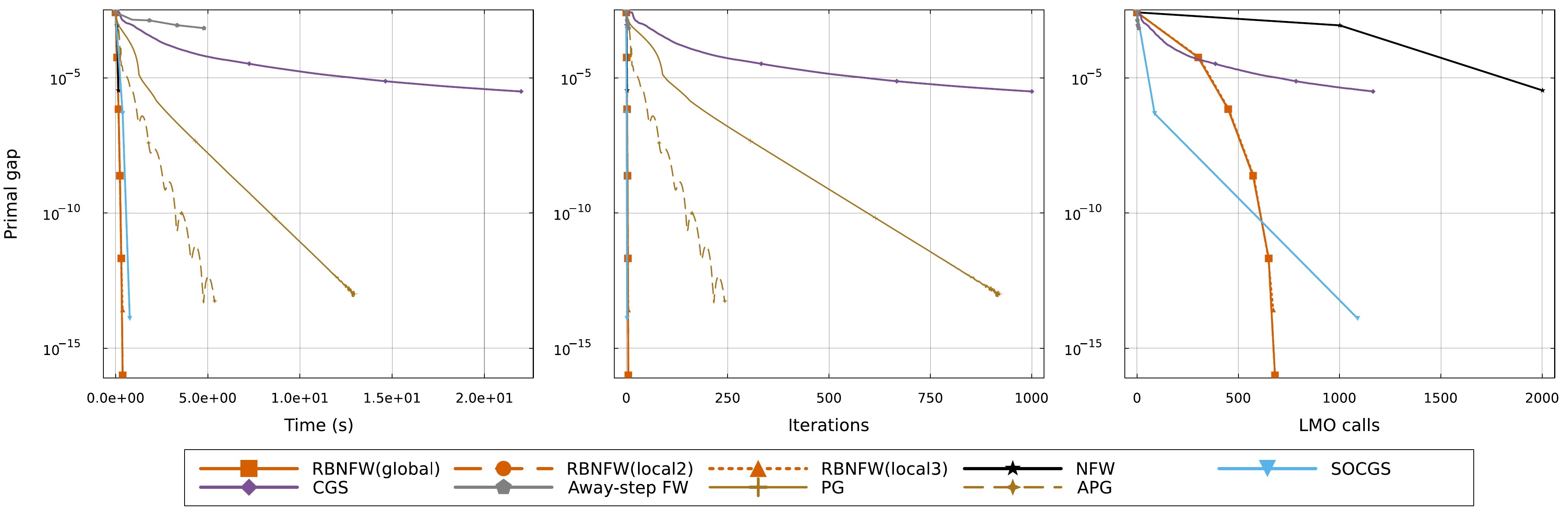}
    \caption{Primal gap for the covtype.binary dataset with a sparse polytope constraint using \cref{alg:dnfw_away}.}
    \label{fig:logistic_covtype_polytope}

    \smallskip
    \includegraphics[width=\textwidth]{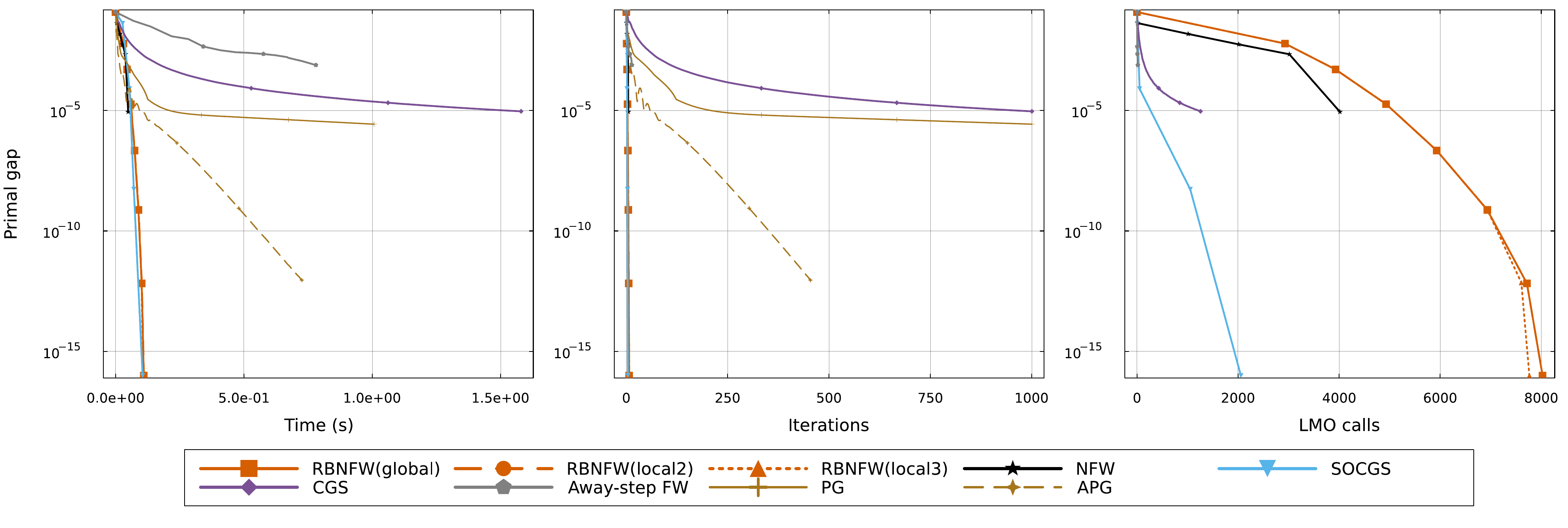}
    \caption{Primal gap for the a9a dataset with a sparse polytope constraint using \cref{alg:dnfw_away}.}
    \label{fig:logistic_a9a_polytope}
\end{figure}

\section{Conclusion and Future Work}

\paragraph{Conclusion.}

We developed damped Newton FW methods for convex optimization over compact convex sets.
Our methods approximately solve constrained damped Newton subproblems using only LMO calls, with either a standard FW or an away-step FW inner loop.
We established global linear convergence with residual backtracking.
Switching to full steps yields local superlinear or quadratic convergence.
A projection-free Newton method can combine global linear convergence with local quadratic convergence over general compact convex sets, without requiring polyhedrality, strict complementarity, or a lower bound on the primal gap for setting the inner accuracy.
In experiments on matrix sensing and ridge-regularized logistic regression, residual backtracking was consistently robust and often achieved the best wall-clock performance among the tested methods.

\paragraph{Future Work.}

Throughout this work, we form each Newton subproblem using the exact Hessian.
An important direction for future work is to accommodate approximate Hessians or to reuse the same Hessian over multiple iterations through lazy updates, thereby reducing the computational cost of constructing and solving the subproblems. 
The proposed methods also require setting at least $\omega$ and $\rho$, and the theoretically admissible range of $\omega$ depends on problem constants and on the choice of $\rho$.
A fully parameter-free method would therefore need to estimate the relevant problem constants adaptively and select the remaining algorithmic parameters automatically.

\subsubsection*{Acknowledgments}
This project has been funded by the Japan Society for the Promotion of Science (JSPS); JSPS KAKENHI Grant Number JP23K28041 and JP25K21156.

\bibliography{main}
\bibliographystyle{plainnat}

\appendix

\section{Additional Related Work}\label{appendix:additional-related-work}
\paragraph{Projection and Proximal Methods.}
For smooth constrained optimization, projected gradient methods are widely used and can be viewed as proximal gradient methods in which the nonsmooth term is the indicator function of the feasible set.
More generally, proximal gradient methods can be interpreted as forward--backward splitting schemes for minimizing the sum of a smooth function and a proximable nonsmooth function.
The operator-splitting foundations of these methods trace back to early work on monotone operators and variational inequalities~\citep{Bruck1975-et,Lions1979-id,Passty1979-sk}; see also~\citep{Combettes2005-lf} for a proximal forward--backward formulation in signal recovery.
The fast iterative shrinkage-thresholding algorithm (FISTA)~\citep{Beck2009-kr} is a widely used accelerated proximal gradient method.
Under relative smoothness, the Euclidean distance can be replaced by the Bregman distance, and the Bregman proximal gradient method can be derived in this setting~\citep{Bauschke2017-hg,Bolte2018-zt,Lu2018-ii}.
Bregman proximal methods have further been developed for difference-of-convex optimization~\citep{Takahashi2022-ml} and signal-processing applications~\citep{Takahashi2023-uh,Takahashi2026-rv}.
Approximate Bregman proximal gradient methods, which use second-order information from Bregman distances, have also been studied~\citep{Takahashi2024-ej,Fujiki2025-do}.

\paragraph{Second-Order Methods.}

Classical Newton's method can attain local quadratic convergence under standard regularity conditions while its unit-step iteration need not be globally convergent.
Common globalization mechanisms include damping, line searches, trust regions, and regularization.
For convex objectives with Lipschitz-continuous Hessians, the cubic-regularized Newton method of~\citet{Nesterov2006-cm} achieves a global rate of $\O(k^{-2})$ and admits local quadratic convergence under standard local regularity conditions.
More recent work has shown that similarly fast global guarantees can be obtained using damped or quadratically regularized Newton steps.
In particular,~\citet{Hanzely2022-pn} proposed an explicit damping schedule that achieves a global $\O(k^{-2})$ rate together with local quadratic convergence, while~\citet{Mishchenko2023-rn} established a global $\O(k^{-2})$ rate using gradient-dependent quadratic regularization.
Gradient regularization has also been extended to non-Euclidean geometries using Bregman distances and to composite convex optimization~\citep{Doikov2023-kh}.
The super-universal regularized Newton method~\citep{Doikov2024-yc} adaptively handles several H\"older smoothness classes, and~\citet{Hanzely2026-kv} developed stepsize schedules and universal backtracking strategies for Newton methods under H\"older continuity of the Hessian or the third derivative.
The latter work provides the main unconstrained motivation for the stepsize strategies developed in this paper.
For composite optimization, proximal Newton methods minimize a quadratic model of the smooth term together with the nonsmooth term and preserve Newton-type local convergence with inexact subproblem solutions under suitable regularity and inner accuracy conditions~\citep{Lee2014-zv}.
This framework includes convex constraints by taking the nonsmooth term to be the indicator function of the feasible set.
Our contribution is an LMO-based framework that combines computable FW-gap-based inner stopping criteria, residual backtracking, and a switch to full steps after finitely many iterations to establish global linear and fast local convergence over general compact convex sets.
Contracting proximal methods instead minimize contracted objective models augmented by Bregman regularization and allow inexact solutions of the resulting auxiliary subproblems~\citep{Doikov2020-zi}.
Closely related to our setting, the affine-invariant contracting-point framework of~\citet{Doikov2023-eo} recovers the FW method in its first-order instance and yields an inexact Contracting Newton method in its second-order instance.
This framework applies to bounded domains and admits an implementation based on LMOs, but its globalization mechanism contracts the domain, whereas our framework damps the quadratic model.
In addition to NFW and SOCGS,~\citet{Goncalves2017-me} proposed the Newton conditional gradient method, and \citet{Wang2025-gs} proposed the cubic-regularized Newton method for self-concordant finite-mixture objectives with polyhedral constraints.
Their cubic subproblems are solved using away-step FW methods, and their analysis establishes a global $\O(k^{-2})$ rate and local quadratic convergence.
Another line of research reduces the cost of second-order information by reusing or approximating the Hessian.
\citet{Doikov2023-vt} proposed lazy Hessian updates that reuse the same Hessian over several iterations, while~\citet{Semenov2026-ve} developed gradient-regularized methods with approximate Hessians.

\section{Proofs of Section~\ref{sec:proposed-method}}\label{appendix:proofs-proposed-method}
\subsection{Proof of Proposition~\ref{proposition:lmo_complexity}}\label{appendix:proof-lmo_complexity}
\begin{proof}[Proof of \cref{proposition:lmo_complexity}]
    If $D=0$, then $\X$ is a singleton and no inner iteration is required.
    Hence, we assume that $D>0$.

    For any $u,v\in\X$ and $\gamma\in[0,1]$, we have
    \begin{align*}
        \Psi_{\alpha_k}(u+\gamma(v-u);x_k)
        &=\Psi_{\alpha_k}(u;x_k)+\gamma\left\langle\nabla\Psi_{\alpha_k}(u;x_k),v-u\right\rangle+\frac{\gamma^2}{2\alpha_k}\|v-u\|_{x_k}^2\\
        &\leq\Psi_{\alpha_k}(u;x_k)+\gamma\left\langle\nabla\Psi_{\alpha_k}(u;x_k),v-u\right\rangle+\frac{2\gamma^2D^2}{\alpha_k},
    \end{align*}
    where we used
    \begin{align*}
        \|v-u\|_{x_k}\leq\|v-x_k\|_{x_k}+\|u-x_k\|_{x_k}\leq2D.
    \end{align*}
    Since the stepsize $\gamma_t$ in \cref{alg:fw} exactly minimizes $\Psi_{\alpha_k}(u;x_k)$ along the segment joining $u_t$ and $v_t$, for any $\gamma\in[0,1]$,
    \begin{align}
        \Psi_{\alpha_k}(u_{t+1};x_k)\leq\Psi_{\alpha_k}(u_t;x_k)-\gamma G_t+\frac{2\gamma^2D^2}{\alpha_k}, \label{ineq:fw_lmo_descent}
    \end{align}
    where $G_t$ is the FW gap of $\Psi_{\alpha_k}(\cdot;x_k)$ over $\X$ at $u_t$ (see \eqref{ineq:fw_termination}).
    Moreover, the convexity of $\Psi_{\alpha_k}(\cdot;x_k)$ gives
    \begin{align}
        \Psi_{\alpha_k}(u_t;x_k)-\min_{u\in\X}\Psi_{\alpha_k}(u;x_k)\leq G_t. \label{ineq:fw_primal_gap_fw_gap}
    \end{align}

    Suppose that \cref{alg:fw} does not terminate at $t=0$.
    Substituting the comparison value $\gamma=1$ into \eqref{ineq:fw_lmo_descent} and using \eqref{ineq:fw_primal_gap_fw_gap}, we obtain
    \begin{align}
        \Psi_{\alpha_k}(u_1;x_k)-\min_{u\in\X}\Psi_{\alpha_k}(u;x_k)\leq\frac{2D^2}{\alpha_k}. \label{ineq:fw_first_primal_gap}
    \end{align}

    We next show that, for every $t\geq1$,
    \begin{align}
        \Psi_{\alpha_k}(u_t;x_k)-\min_{u\in\X}\Psi_{\alpha_k}(u;x_k)\leq\frac{8D^2}{\alpha_k(t+1)}. \label{ineq:fw_primal_gap_rate}
    \end{align}
    The case $t=1$ follows from \eqref{ineq:fw_first_primal_gap}.
    Suppose that \eqref{ineq:fw_primal_gap_rate} holds for some $t\geq1$.
    Substituting the comparison value $\gamma=2/(t+2)$ into \eqref{ineq:fw_lmo_descent} and using \eqref{ineq:fw_primal_gap_fw_gap}, we obtain
    \begin{align*}
        \Psi_{\alpha_k}(u_{t+1};x_k)-\min_{u\in\X}\Psi_{\alpha_k}(u;x_k)
        &=\min_{\gamma\in[0,1]}\Psi_{\alpha_k}((1-\gamma)u_t+\gamma v_t;x_k)-\min_{u\in\X}\Psi_{\alpha_k}(u;x_k)\\
        &\leq\Psi_{\alpha_k}\left(\left(1-\frac{2}{t+2}\right)u_t+\frac{2}{t+2}v_t;x_k\right)-\min_{u\in\X}\Psi_{\alpha_k}(u;x_k)\\
        &=\Psi_{\alpha_k}(u_t;x_k)-\min_{u\in\X}\Psi_{\alpha_k}(u;x_k)-\frac{2}{t+2}G_t\\
        &\quad+\frac{2}{\alpha_k(t+2)^2}\|v_t-u_t\|_{x_k}^2\\
        &\leqby{\eqref{ineq:fw_primal_gap_fw_gap}}\frac{t}{t+2}\left(\Psi_{\alpha_k}(u_t;x_k)-\min_{u\in\X}\Psi_{\alpha_k}(u;x_k)\right)\\
        &\quad+\frac{2}{\alpha_k(t+2)^2}\|v_t-u_t\|_{x_k}^2\\
        &\leq\frac{t}{t+2}\frac{8D^2}{\alpha_k(t+1)}+\frac{8D^2}{\alpha_k(t+2)^2}\leq\frac{8D^2}{\alpha_k(t+2)}.
    \end{align*}
    The last inequality follows from
    \begin{align*}
        \frac{t}{(t+1)(t+2)}+\frac{1}{(t+2)^2}\leq\frac{1}{t+2}.
    \end{align*}
    Therefore, \eqref{ineq:fw_primal_gap_rate} follows by induction.

    Let $U\geq1$ be an integer.
    If \cref{alg:fw} terminates within the first $U > 0$ inner updates, then the conclusion already holds.
    Hence, suppose that $u_1,\dots,u_{U+1}$ are generated.
    Setting $\gamma=2/(t+2)$ in \eqref{ineq:fw_lmo_descent}, for every $t=1,\dots,U$, we obtain
    \begin{align*}
        \frac{2}{t+2}G_t\leq\Psi_{\alpha_k}(u_t;x_k)-\Psi_{\alpha_k}(u_{t+1};x_k)+\frac{8D^2}{\alpha_k(t+2)^2}.
    \end{align*}
    Multiplying both sides by $(t+1)(t+2)/2$ and summing over $t=1,\dots,U$, we obtain
    \begin{align*}
        \sum_{t=1}^U(t+1)G_t
        &\leq\sum_{t=1}^U\frac{(t+1)(t+2)}{2}\left(\Psi_{\alpha_k}(u_t;x_k)-\Psi_{\alpha_k}(u_{t+1};x_k)\right)+\frac{4D^2}{\alpha_k}\sum_{t=1}^U\frac{t+1}{t+2}\\
        &=3\left(\Psi_{\alpha_k}(u_1;x_k)-\min_{u\in\X}\Psi_{\alpha_k}(u;x_k)\right)\\
        &\quad+\sum_{t=2}^U\left(\frac{(t+1)(t+2)}{2} - \frac{t(t+1)}{2}\right)\left(\Psi_{\alpha_k}(u_t;x_k)-\min_{u\in\X}\Psi_{\alpha_k}(u;x_k)\right)\\
        &\quad-\frac{(U+1)(U+2)}{2}\left(\Psi_{\alpha_k}(u_{U+1};x_k)-\min_{u\in\X}\Psi_{\alpha_k}(u;x_k)\right)+\frac{4D^2}{\alpha_k}\sum_{t=1}^U\frac{t+1}{t+2}\\
        &=3\left(\Psi_{\alpha_k}(u_1;x_k)-\min_{u\in\X}\Psi_{\alpha_k}(u;x_k)\right)\\
        &\quad+\sum_{t=2}^U(t+1)\left(\Psi_{\alpha_k}(u_t;x_k)-\min_{u\in\X}\Psi_{\alpha_k}(u;x_k)\right)\\
        &\quad-\frac{(U+1)(U+2)}{2}\left(\Psi_{\alpha_k}(u_{U+1};x_k)-\min_{u\in\X}\Psi_{\alpha_k}(u;x_k)\right)+\frac{4D^2}{\alpha_k}\sum_{t=1}^U\frac{t+1}{t+2}\\
        &\leqby{(a)}3\frac{2D^2}{\alpha_k}+\sum_{t=2}^U(t+1)\frac{8D^2}{\alpha_k(t+1)}+\frac{4D^2}{\alpha_k}\sum_{t=1}^U\frac{t+1}{t+2}\\
        &\leq\frac{6D^2}{\alpha_k}+\frac{8(U-1)D^2}{\alpha_k}+\frac{4UD^2}{\alpha_k}=\frac{(12U-2)D^2}{\alpha_k}\leq\frac{12UD^2}{\alpha_k},
    \end{align*}
    where (a) follows from \eqref{ineq:fw_first_primal_gap} and \eqref{ineq:fw_primal_gap_rate}.
    Consequently,
    \begin{align*}
        \min_{1\leq t\leq U}G_t\frac{U(U+3)}{2}\leq\sum_{t=1}^U(t+1)G_t\leq\frac{12UD^2}{\alpha_k},
    \end{align*}
    and hence
    \begin{align}
        \min_{1\leq t\leq U}G_t\leq\frac{24D^2}{\alpha_k(U+3)}. \label{ineq:fw_min_fw_gap}
    \end{align}

    If \cref{alg:fw} does not terminate at $t=0$, set
    \begin{align*}
        U=\left\lceil\frac{24D^2}{\alpha_k\eta_k}\right\rceil.
    \end{align*}
    Then, by \eqref{ineq:fw_min_fw_gap}, there exists some $t\in\{1,\dots,U\}$ such that $G_t\leq\eta_k$.
    Therefore, \cref{alg:fw} terminates after at most $U$ inner updates.
\end{proof}

\subsection{Proof of Proposition~\ref{proposition:lmo_complexity_away}}\label{appendix:proof-lmo_complexity_away}

We present $\afw$ of \cref{alg:dnfw_away} in \cref{alg:afw}.

\begin{algorithm}[!htb]
    \caption{$\texttt{AFW}(h, H, u, \eta, \Scal, \lambda)$}\label{alg:afw}
    $u_0 \gets u$, $\Scal_0 \gets \Scal$, $\lambda_0 \gets \lambda$\;
    \For{$t=0,1,2,\dots$}{
        $v_t^{\fwv}\gets\argmin_{v\in\X}\langle h+H(u_t-u),v\rangle$\;
        \If{$\langle h + H(u_t - u), u_t - v_t^{\fwv}\rangle\leq\eta$}{
            \Return $(u_t,\Scal_t,\lambda_t)$\;
        }
        $v_t^{\away}\gets\argmax_{v\in\Scal_t}\langle h+H(u_t-u),v\rangle$\;
        \If{$\langle h + H(u_t - u), u_t - v_t^{\fwv}\rangle\geq\langle h + H(u_t - u), v_t^{\away} - u_t\rangle$}{
            $d_t \gets u_t - v_t^{\fwv}$, $\gamma_{t,\max} \gets 1$ \Comment*{FW step}
        }
        \Else{
            $d_t \gets v_t^{\away} - u_t$, $\gamma_{t,\max} \gets \frac{\lambda_{v_t^{\away}, t}}{1 - \lambda_{v_t^{\away}, t}}$ \Comment*{Away step}
        }
        $\gamma_t \gets \min\left\{\gamma_{t,\max}, \frac{\langle h + H(u_t - u), d_t \rangle}{\|d_t\|_H^2}\right\}$\label{eq:afw_exact_stepsize}\;
        $u_{t+1} \gets u_t - \gamma_t d_t$\;
        \If{$\langle h + H(u_t - u), u_t - v_t^{\fwv} \rangle \geq \langle h + H(u_t - u), v_t^{\away} - u_t\rangle$}{
            $\lambda_{v,t+1} \gets (1 - \gamma_t)\lambda_{v,t}$ for all $v \in \Scal_t\setminus\{v_t^{\fwv}\}$ \Comment*{FW step}
            $\lambda_{v_t^{\fwv}, t+1} \gets \begin{cases}
                \gamma_t & \text{if } v_t^{\fwv} \not\in \Scal_t\\
                (1 - \gamma_t)\lambda_{v_t^{\fwv},t} + \gamma_t & \text{if } v_t^{\fwv} \in \Scal_t
            \end{cases}$\;
            $\Scal_{t+1} \gets \begin{cases}
                \Scal_t \cup \{v_t^{\fwv}\} & \text{if } \gamma_t < 1\\
                \{v_t^{\fwv}\} & \text{if } \gamma_t = 1
            \end{cases}$\;
        }\Else{
            $\lambda_{v,t+1}\gets(1 + \gamma_t)\lambda_{v,t}$ for $v\in\Scal_t\setminus\{v_t^{\away}\}$ \Comment*{Away step}
            $\lambda_{v_t^{\away},t+1} \gets (1 + \gamma_t)\lambda_{v_t^{\away},t} - \gamma_t$\;
            $\Scal_{t+1} \gets \begin{cases}
                \Scal_t & \text{if } \lambda_{v_t^{\away},t+1} > 0\\
                \Scal_t\setminus\{v_t^{\away}\} & \text{if } \lambda_{v_t^{\away},t+1} = 0
            \end{cases}$\;
        }
    }
\end{algorithm}

The weight vector $\lambda_t=(\lambda_{v,t})_{v\in\Scal_t}$ satisfies $u_t=\sum_{v\in\Scal_t}\lambda_{v,t}v$, $\lambda_{v,t}>0$, and $\sum_{v\in\Scal_t}\lambda_{v,t}=1$.

\begin{proof}[Proof of \cref{proposition:lmo_complexity_away}]
    By the direction-selection rule of \cref{alg:afw}, we have
    \begin{align*}
        \langle\nabla\Psi_{\alpha_k}(u_t;x_k),d_t\rangle\geq\frac{1}{2}\langle\nabla\Psi_{\alpha_k}(u_t;x_k),v_t^{\mathrm A}-v_t^{\mathrm{FW}}\rangle.
    \end{align*}
    The geometric strong convexity inequality associated with the pyramidal width~\cite[Lemma~2.27]{Braun2025-sz} provides
    \begin{align}
        \Psi_{\alpha_k}(u_t;x_k)-\min_{u\in\X}\Psi_{\alpha_k}(u;x_k)\leq\frac{\langle\nabla\Psi_{\alpha_k}(u_t;x_k),v_t^{\mathrm A}-v_t^{\mathrm{FW}}\rangle^2}{2\mu_k\delta_{\X}^2}\leq\frac{2\langle\nabla\Psi_{\alpha_k}(u_t;x_k),d_t\rangle^2}{\mu_k\delta_{\X}^2}. \label{ineq:afw_geometric_strong_convexity}
    \end{align}
    For any $\gamma\in[0,\gamma_{t,\max}]$, the quadratic structure of $\Psi_{\alpha_k}(\cdot;x_k)$ yields
    \begin{align}
        \Psi_{\alpha_k}(u_t-\gamma d_t;x_k)=\Psi_{\alpha_k}(u_t;x_k)-\gamma\langle\nabla\Psi_{\alpha_k}(u_t;x_k),d_t\rangle+\frac{\gamma^2}{2\alpha_k}\|d_t\|_{x_k}^2. \label{eq:afw_quadratic_expansion}
    \end{align}

    Suppose first that the current update is an away step with $\gamma_t<\gamma_{t,\max}$ or an FW step with $\gamma_t<1$.
    By Line~\ref{eq:afw_exact_stepsize} of \cref{alg:afw},
    \begin{align*}
        \Psi_{\alpha_k}(u_t;x_k)-\Psi_{\alpha_k}(u_{t+1};x_k)=\frac{\alpha_k\langle\nabla\Psi_{\alpha_k}(u_t;x_k),d_t\rangle^2}{2\|d_t\|_{x_k}^2}\geq\frac{\langle\nabla\Psi_{\alpha_k}(u_t;x_k),d_t\rangle^2}{2L_kd_{\X}^2},
    \end{align*}
    where we used $\|d_t\|_{x_k}^2/\alpha_k\leq L_k\|d_t\|^2\leq L_kd_{\X}^2$. Combining this inequality with \eqref{ineq:afw_geometric_strong_convexity}, we obtain
    \begin{align*}
        \Psi_{\alpha_k}(u_t;x_k)-\Psi_{\alpha_k}(u_{t+1};x_k)\geq\frac{\mu_k\delta_{\X}^2}{4L_kd_{\X}^2}\left(\Psi_{\alpha_k}(u_t;x_k)-\min_{u\in\X}\Psi_{\alpha_k}(u;x_k)\right).
    \end{align*}

    Next, suppose that the current step is an FW step with $\gamma_t=1$. By Line~\ref{eq:afw_exact_stepsize} of \cref{alg:afw},
    \begin{align*}
        \frac{\alpha_k\langle\nabla\Psi_{\alpha_k}(u_t;x_k),d_t\rangle}{\|d_t\|_{x_k}^2}\geq1.
    \end{align*}
    Hence, \eqref{eq:afw_quadratic_expansion} implies
    \begin{align*}
        \Psi_{\alpha_k}(u_t;x_k)-\Psi_{\alpha_k}(u_{t+1};x_k)&\geq\frac{1}{2}\langle\nabla\Psi_{\alpha_k}(u_t;x_k),d_t\rangle=\frac{G_t}{2}\\
        &\geq\frac{1}{2}\left(\Psi_{\alpha_k}(u_t;x_k)-\min_{u\in\X}\Psi_{\alpha_k}(u;x_k)\right),
    \end{align*}
    where the last inequality follows from \eqref{ineq:fw_primal_gap_fw_gap} due to the convexity of $\Psi_{\alpha_k}(\cdot;x_k)$.
    Since $\mu_k\leq L_k$ and $\delta_{\X}\leq d_{\X}$, we have $\mu_k\delta_{\X}^2/(4L_kd_{\X}^2)\leq1/4$. Therefore, in either case,
    \begin{align}
        \Psi_{\alpha_k}(u_{t+1};x_k)-\min_{u\in\X}\Psi_{\alpha_k}(u;x_k)\leq\left(1-\frac{\mu_k\delta_{\X}^2}{4L_kd_{\X}^2}\right)\left(\Psi_{\alpha_k}(u_t;x_k)-\min_{u\in\X}\Psi_{\alpha_k}(u;x_k)\right). \label{ineq:afw_non_drop_contraction}
    \end{align}
    Thus, \eqref{ineq:afw_non_drop_contraction} holds for every non-drop step. For a drop step, exact line minimization gives $\Psi_{\alpha_k}(u_{t+1};x_k)\leq\Psi_{\alpha_k}(u_t;x_k)$.

    Each drop step removes one vertex from the active set, whereas each FW step adds at most one vertex. Hence, among the first $U$ inner updates, the number of drop steps is at most the number of FW steps plus $|\widehat{\Scal}_k|-1$. Since every FW step is a non-drop step, the number of non-drop steps is at least $\max\{0,\lceil(U-|\widehat{\Scal}_k|+1)/2\rceil\}$.
    Consequently,
    \begin{align}
        \Psi_{\alpha_k}(u_U;x_k)-\min_{u\in\X}\Psi_{\alpha_k}(u;x_k)
        &\leq\left(\Psi_{\alpha_k}(u_0;x_k)-\min_{u\in\X}\Psi_{\alpha_k}(u;x_k)\right)\notag\\
        &\quad\times\left(1-\frac{\mu_k\delta_{\X}^2}{4L_kd_{\X}^2}\right)^{\max\{0,\lceil(U-|\widehat{\Scal}_k|+1)/2\rceil\}}. \label{ineq:afw_primal_gap_contraction}
    \end{align}

    For every $\gamma\in[0,1]$, we have
    \begin{align*}
        \Psi_{\alpha_k}(u_t;x_k)-\min_{u\in\X}\Psi_{\alpha_k}(u;x_k)
        &\geq\Psi_{\alpha_k}(u_t;x_k)-\Psi_{\alpha_k}\bigl(u_t+\gamma(v_t^{\mathrm{FW}}-u_t);x_k\bigr)\\
        &\geq\gamma G_t-\frac{L_k\gamma^2d_{\X}^2}{2}.
    \end{align*}
    Suppose that $\Psi_{\alpha_k}(u_t;x_k)-\min_{u\in\X}\Psi_{\alpha_k}(u;x_k)\leq L_kd_{\X}^2/2$. If $G_t>L_kd_{\X}^2$, then choosing $\gamma=1$ in the preceding inequality gives a contradiction. Hence, $G_t\leq L_kd_{\X}^2$, and choosing $\gamma=G_t/(L_kd_{\X}^2)$ yields
    \begin{align*}
        G_t\leq d_{\X}\sqrt{2L_k\left(\Psi_{\alpha_k}(u_t;x_k)-\min_{u\in\X}\Psi_{\alpha_k}(u;x_k)\right)}.
    \end{align*}
    Therefore, $\Psi_{\alpha_k}(u_t;x_k)-\min_{u\in\X}\Psi_{\alpha_k}(u;x_k)\leq\zeta_k$ implies $G_t\leq d_{\X}\sqrt{2L_k\zeta_k}\leq\eta_k$.

    Finally, set
    \begin{align*}
        U\coloneq|\widehat{\Scal}_k|-1+2\left\lceil\frac{\log\max\{1,g_k/\zeta_k\}}{-\log\left(1-\mu_k\delta_{\X}^2/(4L_kd_{\X}^2)\right)}\right\rceil.
    \end{align*}
    Since $u_0=x_k$ and the quadratic term in $\Psi_{\alpha_k}(\cdot;x_k)$ is nonnegative, we have the initial model primal-gap bound as follows:
    \begin{align*}
        \Psi_{\alpha_k}(x_k;x_k)-\min_{u\in\X}\Psi_{\alpha_k}(u;x_k)\leq\max_{u\in\X}\langle\nabla f(x_k),x_k-u\rangle=g_k.
    \end{align*}
    If the inner loop has not terminated earlier, then \eqref{ineq:afw_primal_gap_contraction} and the initial model primal-gap bound give
    \begin{align*}
        \Psi_{\alpha_k}(u_U;x_k)-\min_{u\in\X}\Psi_{\alpha_k}(u;x_k)
        &\leq g_k\left(1-\frac{\mu_k\delta_{\X}^2}{4L_kd_{\X}^2}\right)^{\left\lceil\frac{\log\max\{1,g_k/\zeta_k\}}{-\log\left(1-\mu_k\delta_{\X}^2/(4L_kd_{\X}^2)\right)}\right\rceil}\leq\zeta_k.
    \end{align*}
    Hence, $G_U\leq\eta_k$, and the inner loop terminates after at most $U$ iterations.
\end{proof}

\section{Proofs of Section~\ref{subsec:global-convergence}}\label{appendix:proofs-global-convergence}

\paragraph{Convergence Terminology.}
For a nonnegative sequence $b_k\to0$, we say that it converges Q-linearly if $b_{k+1}\leq\psi b_k$ eventually for some $\psi\in(0,1)$, and Q-converges with order at least $l>1$ if $b_{k+1}\leq C b_k^l$ eventually for some $C>0$.
We say that $b_k$ converges R-linearly if $b_k\leq C\psi^k$ for some $C>0$ and $\psi\in(0,1)$.
Q-quadratic convergence means Q-order at least two.
For convergence of the iterates and the primal gap, these definitions apply to the scalar sequences $\|x_k-x^\star\|$ and $f(x_k)-f(x^\star)$, respectively.

\subsection{Proof of Lemma~\ref{lemma:q_r_relation_inexact}}\label{appendix:proof-q_r_relation_inexact}
From the definition of $w_k$, we have
\begin{align}
    x_{k+1}-x_k = -\alpha_k\nabla^2 f(x_k)^{-1}w_k,
    \qquad
    \|x_{k+1}-x_k\|_{x_k} = \alpha_k\|w_k\|_{x_k}^*. \label{eq:scaled_newton_direction}
\end{align}
\begin{proof}[Proof of~\cref{lemma:q_r_relation_inexact}]
    For $k\geq1$, applying \eqref{ineq:fw_termination} at iteration $k$ with $v=x_k$ and at iteration $k-1$ with $v=x_{k+1}$ gives
    \begin{align*}
        \langle s_k,x_k-x_{k+1}\rangle\leq\eta_k,
        \qquad
        \langle s_{k-1},x_{k+1}-x_k\rangle\leq\eta_{k-1}.
    \end{align*}
    The same argument applies when $k=0$ by setting $s_{-1}=0$ and $\eta_{-1}=0$.
    Therefore,
    \begin{align*}
        \langle s_k-s_{k-1},x_{k+1}-x_k\rangle\geq-(\eta_k+\eta_{k-1}).
    \end{align*}
    Using $s_k-s_{k-1}=w_k-r_k$ and \eqref{eq:scaled_newton_direction}, we obtain
    \begin{align*}
        \|w_k\|_{x_k}^{*2}
        &\leq\left\langle r_k,\nabla^2f(x_k)^{-1}w_k\right\rangle+\frac{\eta_k+\eta_{k-1}}{\alpha_k}\\
        &\leq\|r_k\|_{x_k}^*\|w_k\|_{x_k}^*+\frac{\eta_k+\eta_{k-1}}{\alpha_k},
    \end{align*}
    where the second inequality follows from the Cauchy--Schwarz inequality.
    This proves \eqref{ineq:q_r_inexact_basic}.

    Suppose now that \eqref{eq:global_inner_accuracy} holds at every iteration and that $\alpha_k\geq(1+B_q\Delta_k^{\frac{q-2}{q-1}})^{-1}$.
    For $k\geq1$ with $\Delta_{k-1}>0$, set $t=\Delta_k/\Delta_{k-1}\geq\rho$.
    Then
    \begin{align}
        \frac{\eta_{k-1}}{\eta_k}=\frac{\Delta_{k-1}^2(1+B_q\Delta_k^{\frac{q-2}{q-1}})}{\Delta_k^2(1+B_q\Delta_{k-1}^{\frac{q-2}{q-1}})}=t^{-2}\frac{1+B_qt^{\frac{q-2}{q-1}}\Delta_{k-1}^{\frac{q-2}{q-1}}}{1+B_q\Delta_{k-1}^{\frac{q-2}{q-1}}}\leq\rho^{-2}. \label{ineq:alpha_delta_ratio}
    \end{align}
    Indeed, the last fraction is at most $1$ when $t\in[\rho,1]$ and at most $t^{\frac{q-2}{q-1}}$ when $t\geq1$.
    In the latter case, the entire expression is at most $t^{-\frac{q}{q-1}}\leq1$.
    Thus, $\eta_{k-1}\leq\rho^{-2}\eta_k$; the same inequality holds for $k=0$ or $\Delta_{k-1}=0$, since then $\eta_{k-1}=0$.
    Consequently,
    \begin{align*}
        \frac{\eta_k+\eta_{k-1}}{\alpha_k}\leq(1+\rho^{-2})\left(1+B_q\Delta_k^{\frac{q-2}{q-1}}\right)\eta_k=\omega(1+\rho^{-2})\Delta_k^2.
    \end{align*}
    Combining this with \eqref{ineq:q_r_inexact_basic} and $\|r_k\|_{x_k}^*\leq\Delta_k$ gives
    \begin{align*}
        \|w_k\|_{x_k}^{*2}\leq\Delta_k\|w_k\|_{x_k}^*+\omega(1+\rho^{-2})\Delta_k^2.
    \end{align*}
    Solving this quadratic inequality yields $\|w_k\|_{x_k}^*\leq\kappa_\omega\Delta_k$.
    The same argument applies to each trial after replacing $(\alpha_k,x_{k+1},s_k,w_k)$ by $(\alpha_{k,j},x_{k,j},s_{k,j},w_{k,j})$, while keeping the previous accepted $s_{k-1}$ and $\eta_{k-1}$ fixed.
    In particular,
    \begin{align}
        \|w_{k,j}\|_{x_k}^{*2}\leq\|r_k\|_{x_k}^*\|w_{k,j}\|_{x_k}^*+\frac{\eta_k+\eta_{k-1}}{\alpha_{k,j}}\leq\Delta_k\|w_{k,j}\|_{x_k}^*+\omega(1+\rho^{-2})\Delta_k^2, \label{ineq:backtracking_q_bound}
    \end{align}
    because every RBNFW trial satisfies $\alpha_{k,j}\geq(1+B_q\Delta_k^{\frac{q-2}{q-1}})^{-1}$.
\end{proof}

\subsection{A Sufficient Condition for Acceptance}\label{appendix:residual_acceptance}
We now present a sufficient condition for the acceptance of a trial step in RBNFW.
\begin{proposition}\label[proposition]{proposition:residual_acceptance}
    Let $\|w_k\|_{x_k}^*>0$, $p\in\{2,3\}$, $\nu\in[0,1]$, $q=p+\nu$, and $\alpha_k=1/(1+\theta_k)$ with $\theta_k>0$.
    Suppose that \cref{assumption:holder} holds and that $x_{k+1}\in\X$ satisfies \eqref{ineq:fw_termination}.
    If
    \begin{align}
        \theta_k\geq(9L_{p,\nu})^{\frac{1}{q-1}}\|w_k\|_{x_k}^{*\frac{q-2}{q-1}}, \label{ineq:unified_theta_condition}
    \end{align}
    then
    \begin{align}
        \langle\nabla f(x_{k+1})+s_k,\nabla^2f(x_k)^{-1}w_k\rangle\geq\frac{c_p}{\alpha_k\theta_k}\|\nabla f(x_{k+1})+s_k\|_{x_k}^{*2}. \label{ineq:prod_res_dir_qk}
    \end{align}
\end{proposition}
\begin{proof}
    By \cref{lemma:unified_theta_implies_individual}, \eqref{ineq:unified_theta_condition} implies the assumptions of \cref{lemma:prod_res_dir}.
    Applying that lemma and using $1-\alpha_k=\alpha_k\theta_k$ yields \eqref{ineq:prod_res_dir_qk}.
\end{proof}

We first establish three auxiliary lemmas.

\begin{lemma}\label[lemma]{lemma:unified_theta_implies_individual}
    Let $p\in\{2,3\}$, $\nu\in[0,1]$, and $q=p+\nu$, and let $\alpha_k=1/(1+\theta_k)$ with $\theta_k>0$.
    Suppose also that $\|w_k\|_{x_k}^*>0$ and that \eqref{ineq:unified_theta_condition} holds.
    Then, the following statements hold.
    \begin{enumerate}
        \item If $p=2$, then
        \begin{align*}
            \theta_k\geq\frac{L_{2,\nu}}{1+\nu}\alpha_k^\nu\|w_k\|_{x_k}^{*\nu}.
        \end{align*}
        \item If $p=3$, then
        \begin{align*}
            \theta_k\geq\alpha_k\|w_k\|_{x_k}^*
            \max\left\{6\left(\frac{L_{3,\nu}}{1+\nu}\right)^{\frac{1}{1+\nu}},\frac{\sqrt{3}L_{3,\nu}}{(1+\nu)(2+\nu)}\alpha_k^\nu\|w_k\|_{x_k}^{*\nu}\right\}.
        \end{align*}
    \end{enumerate}
\end{lemma}

\begin{proof}
    (i) The case $p=2$:
    Since $q=2+\nu$, \eqref{ineq:unified_theta_condition} gives
    \begin{align*}
        \theta_k^{1+\nu}\geq9L_{2,\nu}\|w_k\|_{x_k}^{*\nu}.
    \end{align*}
    Therefore, using $\alpha_k\theta_k\leq1$, we obtain
    \begin{align*}
        \frac{L_{2,\nu}}{1+\nu}\alpha_k^\nu\|w_k\|_{x_k}^{*\nu}
        \leq\frac{\alpha_k^\nu\theta_k^{1+\nu}}{9(1+\nu)}
        =\frac{(\alpha_k\theta_k)^\nu}{9(1+\nu)}\theta_k
        \leq\theta_k.
    \end{align*}

    (ii) The case $p=3$:
    Since $q=3+\nu$, we have
    \begin{align}
        \theta_k^{2+\nu}\geq9L_{3,\nu}\|w_k\|_{x_k}^{*(1+\nu)}. \label{ineq:theta_power_p3}
    \end{align}
    First,
    \begin{align*}
        \left(6\left(\frac{L_{3,\nu}}{1+\nu}\right)^{\frac{1}{1+\nu}}\alpha_k\|w_k\|_{x_k}^*\right)^{1+\nu}
        \leq\theta_k^{1+\nu}\frac{6^{1+\nu}}{9(1+\nu)}\frac{\theta_k}{(1+\theta_k)^{1+\nu}}
        \leq\theta_k^{1+\nu}\frac{6^{1+\nu}}{9(1+\nu)^2}
        \leq\theta_k^{1+\nu},
    \end{align*}
    where we used
    \begin{align*}
        \frac{\theta_k}{(1+\theta_k)^{1+\nu}}\leq\frac{1}{1+\nu},
        \qquad
        6^{1+\nu}\leq9(1+\nu)^2.
    \end{align*}
    Indeed, Bernoulli's inequality gives $(1+\theta_k)^{1+\nu}\geq1+(1+\nu)\theta_k\geq(1+\nu)\theta_k>0$.
    Taking reciprocals and multiplying by $\theta_k$ yields
    \begin{align*}
        \frac{\theta_k}{(1+\theta_k)^{1+\nu}}
        \leq\frac{\theta_k}{(1+\nu)\theta_k}
        =\frac{1}{1+\nu}.
    \end{align*}
    Moreover, $6^{1+\nu}\leq9(1+\nu)^2$ holds for every $\nu\in[0,1]$.
    Hence,
    \begin{align*}
        \theta_k\geq6\left(\frac{L_{3,\nu}}{1+\nu}\right)^{\frac{1}{1+\nu}}\alpha_k\|w_k\|_{x_k}^*.
    \end{align*}
    In addition, \eqref{ineq:theta_power_p3} and $\alpha_k\theta_k\leq1$ imply
    \begin{align*}
        \frac{\sqrt{3}L_{3,\nu}}{(1+\nu)(2+\nu)}\alpha_k^{1+\nu}\|w_k\|_{x_k}^{*(1+\nu)}
        \leq\theta_k\frac{\sqrt{3}}{9(1+\nu)(2+\nu)}(\alpha_k\theta_k)^{1+\nu}
        \leq\theta_k.
    \end{align*}
    This proves the claim.
\end{proof}

\begin{lemma}\label[lemma]{lemma:holder}
    Suppose that \cref{assumption:holder} holds. If $p=2$, then, for any $x,y\in\X$,
    \begin{align*}
        \|\nabla f(y)-\nabla f(x)-\nabla^2f(x)(y-x)\|_{x}^*\leq\frac{L_{2,\nu}}{1+\nu}\|y-x\|_{x}^{1+\nu}.
    \end{align*}
    If $p=3$, then, for any $x,y\in\X$,
    \begin{align*}
        \|\nabla f(y)-\nabla f(x)-\nabla^2f(x)(y-x)-\frac{1}{2}\nabla^3f(x)[y-x,y-x]\|_{x}^*\leq\frac{L_{3,\nu}}{(1+\nu)(2+\nu)}\|y-x\|_{x}^{2+\nu}.
    \end{align*}
\end{lemma}

\begin{proof}
    Let $d\coloneq y-x$. Since $\X$ is convex, $x+td\in\X$ for every $t\in[0,1]$.
    Suppose first that $\nu>0$.
    If $p=2$, the fundamental theorem of calculus gives
    \begin{align*}
        \nabla f(y)-\nabla f(x)-\nabla^2f(x)d=\int_0^1\left(\nabla^2f(x+td)-\nabla^2f(x)\right)d\,dt.
    \end{align*}
    Hence, by \cref{assumption:holder},
    \begin{align*}
        \|\nabla f(y)-\nabla f(x)-\nabla^2f(x)d\|_x^*\leq L_{2,\nu}\|d\|_x^{1+\nu}\int_0^1t^\nu\,dt=\frac{L_{2,\nu}}{1+\nu}\|d\|_x^{1+\nu}.
    \end{align*}
    If $p=3$, the integral Taylor formula gives
    \begin{align*}
        \nabla f(y)-\nabla f(x)-\nabla^2f(x)d-\frac{1}{2}\nabla^3f(x)[d,d]=\int_0^1(1-t)\left(\nabla^3f(x+td)-\nabla^3f(x)\right)[d,d]\,dt.
    \end{align*}
    Therefore,
    \begin{align*}
        &\|\nabla f(y)-\nabla f(x)-\nabla^2f(x)d-\tfrac{1}{2}\nabla^3f(x)[d,d]\|_x^*\\
        &\qquad\leq L_{3,\nu}\|d\|_x^{2+\nu}\int_0^1(1-t)t^\nu dt=\frac{L_{3,\nu}}{(1+\nu)(2+\nu)}\|d\|_x^{2+\nu}.
    \end{align*}
    When $\nu=0$, the same bounds follow by applying the one-dimensional mean-value theorem for $p=2$ and the second-order Taylor theorem for $p=3$ to $t\mapsto\langle\nabla f(x+td),u\rangle$ for arbitrary $u$ with $\|u\|_x\leq1$, and then taking the supremum over such $u$.
\end{proof}

We next establish the counterpart of~\cite[Lemma~4]{Hanzely2026-kv}.

\begin{lemma}\label[lemma]{lemma:prod_res_dir}
    Let $\|w_k\|_{x_k}^*>0$, and let $\alpha_k=1/(1+\theta_k)$ with $\theta_k>0$.
    \begin{enumerate}
        \item Suppose that \cref{assumption:holder} holds with $p=2$ and that
        \begin{align*}
            \theta_k\geq\frac{L_{2,\nu}}{1+\nu}\alpha_k^\nu\|w_k\|_{x_k}^{*\nu}.
        \end{align*}
        Then,
        \begin{align*}
            \left\langle\nabla f(x_{k+1})+s_k,\nabla^2f(x_k)^{-1}w_k\right\rangle
            \geq\frac{1}{2(1-\alpha_k)}\|\nabla f(x_{k+1})+s_k\|_{x_k}^{*2}.
        \end{align*}
        \item Suppose that \cref{assumption:holder} holds with $p=3$ and that
        \begin{align}
            \theta_k\geq\alpha_k\|w_k\|_{x_k}^*\max\left\{6\left(\frac{L_{3,\nu}}{1+\nu}\right)^{\frac{1}{1+\nu}},\frac{\sqrt{3}L_{3,\nu}}{(1+\nu)(2+\nu)}\alpha_k^\nu\|w_k\|_{x_k}^{*\nu}\right\}. \label{ineq:theta_p3}
        \end{align}
        Then,
        \begin{align*}
            \left\langle\nabla f(x_{k+1})+s_k,\nabla^2f(x_k)^{-1}w_k\right\rangle
            \geq\frac{1}{4(1-\alpha_k)}\|\nabla f(x_{k+1})+s_k\|_{x_k}^{*2}.
        \end{align*}
    \end{enumerate}
\end{lemma}

\begin{proof}
    (i) The case $p=2$:
    By the $L_{2,\nu}$-H\"older continuity, \cref{lemma:holder} implies
    \begin{align*}
        \left(\frac{L_{2,\nu}}{1+\nu}\|x_{k+1}-x_k\|_{x_k}^{1+\nu}\right)^2
        &\geq\|\nabla f(x_{k+1})-\nabla f(x_k)-\nabla^2f(x_k)(x_{k+1}-x_k)\|_{x_k}^{*2}\\
        &=\|\nabla f(x_{k+1})-\nabla f(x_k)+\alpha_kw_k\|_{x_k}^{*2}\\
        &=\|\nabla f(x_{k+1})+s_k-(1-\alpha_k)(\nabla f(x_k)+s_k)\|_{x_k}^{*2}\\
        &=\|\nabla f(x_{k+1})+s_k-(1-\alpha_k)w_k\|_{x_k}^{*2}\\
        &=\|\nabla f(x_{k+1})+s_k\|_{x_k}^{*2}+(1-\alpha_k)^2\|w_k\|_{x_k}^{*2}\\
        &\quad-2(1-\alpha_k)\left\langle\nabla f(x_{k+1})+s_k,\nabla^2f(x_k)^{-1}w_k\right\rangle.
    \end{align*}
    It is therefore sufficient to show that
    \begin{align}
        \frac{1-\alpha_k}{2}\|w_k\|_{x_k}^{*2}
        \geq\frac{1}{2(1-\alpha_k)}
        \left(\frac{L_{2,\nu}}{1+\nu}\|x_{k+1}-x_k\|_{x_k}^{1+\nu}\right)^2. \label{ineq:sigma}
    \end{align}
    By \eqref{eq:scaled_newton_direction} and $1-\alpha_k=\alpha_k\theta_k$, \eqref{ineq:sigma} is equivalent to $\theta_k\geq\frac{L_{2,\nu}}{1+\nu}\alpha_k^\nu\|w_k\|_{x_k}^{*\nu}$, which is equivalent to the assumption of the lemma.
    Consequently,
    \begin{align*}
        \left\langle\nabla f(x_{k+1})+s_k,\nabla^2f(x_k)^{-1}w_k\right\rangle
        &\geq\frac{1}{2(1-\alpha_k)}\|\nabla f(x_{k+1})+s_k\|_{x_k}^{*2}+\frac{1-\alpha_k}{2}\|w_k\|_{x_k}^{*2}\\
        &\quad-\frac{1}{2(1-\alpha_k)}
        \left(\frac{L_{2,\nu}}{1+\nu}\|x_{k+1}-x_k\|_{x_k}^{1+\nu}\right)^2\\
        &\geqby{\eqref{ineq:sigma}}\frac{1}{2(1-\alpha_k)}\|\nabla f(x_{k+1})+s_k\|_{x_k}^{*2}.
    \end{align*}

    (ii) The case $p=3$:
    First, by the $L_{3,\nu}$-H\"older continuity, \cref{lemma:holder} implies
    \begin{align*}
        &\left(\frac{L_{3,\nu}}{(1+\nu)(2+\nu)}\|x_{k+1}-x_k\|_{x_k}^{2+\nu}\right)^2\\
        &\geq\left\|\nabla f(x_{k+1})-\nabla f(x_k)-\nabla^2f(x_k)(x_{k+1}-x_k)-\frac12\nabla^3f(x_k)[x_{k+1}-x_k,x_{k+1}-x_k]\right\|_{x_k}^{*2}\\
        &=\left\|\nabla f(x_{k+1})+s_k-(1-\alpha_k)w_k-\frac12\nabla^3f(x_k)[x_{k+1}-x_k,x_{k+1}-x_k]\right\|_{x_k}^{*2}\\
        &=\|\nabla f(x_{k+1})+s_k\|_{x_k}^{*2}+(1-\alpha_k)^2\|w_k\|_{x_k}^{*2}+\frac14\|\nabla^3f(x_k)[x_{k+1}-x_k,x_{k+1}-x_k]\|_{x_k}^{*2}\\
        &\quad-2(1-\alpha_k)\left\langle\nabla f(x_{k+1})+s_k,\nabla^2f(x_k)^{-1}w_k\right\rangle\\
        &\quad+(1-\alpha_k)\left\langle\nabla^2f(x_k)^{-1/2}w_k,\nabla^2f(x_k)^{-1/2}\nabla^3f(x_k)[x_{k+1}-x_k,x_{k+1}-x_k]\right\rangle\\
        &\quad-\left\langle\nabla^2f(x_k)^{-1/2}(\nabla f(x_{k+1})+s_k),\nabla^2f(x_k)^{-1/2}\nabla^3f(x_k)[x_{k+1}-x_k,x_{k+1}-x_k]\right\rangle\\
        &\geqby{(a)}\frac12\|\nabla f(x_{k+1})+s_k\|_{x_k}^{*2}+(1-\alpha_k)^2\|w_k\|_{x_k}^{*2}-\frac14\|\nabla^3f(x_k)[x_{k+1}-x_k,x_{k+1}-x_k]\|_{x_k}^{*2}\\
        &\quad-2(1-\alpha_k)\left\langle\nabla f(x_{k+1})+s_k,\nabla^2f(x_k)^{-1}w_k\right\rangle\\
        &\quad-(1-\alpha_k)\|w_k\|_{x_k}^*\|\nabla^3f(x_k)[x_{k+1}-x_k,x_{k+1}-x_k]\|_{x_k}^*\\
        &\geqby{(b)}\frac12\|\nabla f(x_{k+1})+s_k\|_{x_k}^{*2}+(1-\alpha_k)^2\|w_k\|_{x_k}^{*2}-\left(\frac{L_{3,\nu}}{1+\nu}\right)^{\frac{2}{1+\nu}}\alpha_k^4\|w_k\|_{x_k}^{*4}\\
        &\quad-2(1-\alpha_k)\left\langle\nabla f(x_{k+1})+s_k,\nabla^2f(x_k)^{-1}w_k\right\rangle\\
        &\quad-2\left(\frac{L_{3,\nu}}{1+\nu}\right)^{\frac{1}{1+\nu}}\alpha_k^2(1-\alpha_k)\|w_k\|_{x_k}^{*3}.
    \end{align*}
    In (a), we used the Cauchy--Schwarz and Young inequalities:
    \begin{align*}
        &-\left\langle\nabla^2f(x_k)^{-1/2}(\nabla f(x_{k+1})+s_k),\nabla^2f(x_k)^{-1/2}\nabla^3f(x_k)[x_{k+1}-x_k,x_{k+1}-x_k]\right\rangle\\
        &\geq-\|\nabla f(x_{k+1})+s_k\|_{x_k}^*\|\nabla^3f(x_k)[x_{k+1}-x_k,x_{k+1}-x_k]\|_{x_k}^*\\
        &\geq-\frac12\|\nabla f(x_{k+1})+s_k\|_{x_k}^{*2}-\frac12\|\nabla^3f(x_k)[x_{k+1}-x_k,x_{k+1}-x_k]\|_{x_k}^{*2},
    \end{align*}
    together with
    \begin{align*}
        &(1-\alpha_k)\left\langle\nabla^2f(x_k)^{-1/2}w_k,\nabla^2f(x_k)^{-1/2}\nabla^3f(x_k)[x_{k+1}-x_k,x_{k+1}-x_k]\right\rangle\\
        &\geq-(1-\alpha_k)\|w_k\|_{x_k}^*\|\nabla^3f(x_k)[x_{k+1}-x_k,x_{k+1}-x_k]\|_{x_k}^*.
    \end{align*}
    In step (b), we used the definition of $\Gamma_{3,\X}$, the normalization of $L_{3,\nu}$ in \cref{assumption:holder}, and \eqref{eq:scaled_newton_direction}:
    \begin{align*}
        \|\nabla^3f(x_k)[x_{k+1}-x_k,x_{k+1}-x_k]\|_{x_k}^*
        &\leq\Gamma_{3,\X}\|x_{k+1}-x_k\|_{x_k}^2\\
        &\leq2\left(\frac{L_{3,\nu}}{1+\nu}\right)^{\frac{1}{1+\nu}}\|x_{k+1}-x_k\|_{x_k}^2\\
        &=2\left(\frac{L_{3,\nu}}{1+\nu}\right)^{\frac{1}{1+\nu}}\alpha_k^2\|w_k\|_{x_k}^{*2}.
    \end{align*}
    Rearranging the above inequality and using~\eqref{eq:scaled_newton_direction} gives
    \begin{align*}
        &\left\langle\nabla f(x_{k+1})+s_k,\nabla^2f(x_k)^{-1}w_k\right\rangle\\
        &\geq\frac{1}{4(1-\alpha_k)}\|\nabla f(x_{k+1})+s_k\|_{x_k}^{*2}+\frac{1-\alpha_k}{2}\|w_k\|_{x_k}^{*2}-\frac12\left(\frac{L_{3,\nu}}{1+\nu}\right)^{\frac{2}{1+\nu}}\frac{\alpha_k^4}{1-\alpha_k}\|w_k\|_{x_k}^{*4}\\
        &\quad-\left(\frac{L_{3,\nu}}{1+\nu}\right)^{\frac{1}{1+\nu}}\alpha_k^2\|w_k\|_{x_k}^{*3}-\frac{1}{2(1-\alpha_k)}\left(\frac{L_{3,\nu}}{(1+\nu)(2+\nu)}\right)^2(\alpha_k\|w_k\|_{x_k}^*)^{2(2+\nu)}.
    \end{align*}
    If $\alpha_k=1/(1+\theta_k)$ satisfies
    \begin{align}
        \frac{1-\alpha_k}{6}\|w_k\|_{x_k}^{*2}
        &\geq\frac12\left(\frac{L_{3,\nu}}{1+\nu}\right)^{\frac{2}{1+\nu}}\frac{\alpha_k^4}{1-\alpha_k}\|w_k\|_{x_k}^{*4}, \label{ineq:theta_p3_alpha4}\\
        \frac{1-\alpha_k}{6}\|w_k\|_{x_k}^{*2}
        &\geq\left(\frac{L_{3,\nu}}{1+\nu}\right)^{\frac{1}{1+\nu}}\alpha_k^2\|w_k\|_{x_k}^{*3}, \label{ineq:theta_p3_alpha2}\\
        \frac{1-\alpha_k}{6}\|w_k\|_{x_k}^{*2}
        &\geq\frac{1}{2(1-\alpha_k)}\left(\frac{L_{3,\nu}}{(1+\nu)(2+\nu)}\right)^2(\alpha_k\|w_k\|_{x_k}^*)^{2(2+\nu)}, \label{ineq:theta_p3_alpha2nu}
    \end{align}
    then
    \begin{align*}
        \left\langle\nabla f(x_{k+1})+s_k,\nabla^2f(x_k)^{-1}w_k\right\rangle
        \geq\frac{1}{4(1-\alpha_k)}\|\nabla f(x_{k+1})+s_k\|_{x_k}^{*2}.
    \end{align*}
    Since $1-\alpha_k=\alpha_k\theta_k$, we have
    \begin{align*}
        1-\alpha_k
        &=\alpha_k\theta_k\\
        &\geq
        \begin{cases}
            \sqrt{3}\left(\frac{L_{3,\nu}}{1+\nu}\right)^{\frac{1}{1+\nu}}\alpha_k^2\|w_k\|_{x_k}^*, &\text{if }\theta_k\geq\sqrt{3}\left(\frac{L_{3,\nu}}{1+\nu}\right)^{\frac{1}{1+\nu}}\alpha_k\|w_k\|_{x_k}^*,\\
            6\left(\frac{L_{3,\nu}}{1+\nu}\right)^{\frac{1}{1+\nu}}\alpha_k^2\|w_k\|_{x_k}^*, &\text{if }\theta_k\geq6\left(\frac{L_{3,\nu}}{1+\nu}\right)^{\frac{1}{1+\nu}}\alpha_k\|w_k\|_{x_k}^*,\\
            \frac{\sqrt{3}L_{3,\nu}}{(1+\nu)(2+\nu)}\alpha_k^{2+\nu}\|w_k\|_{x_k}^{*(1+\nu)}, &\text{if }\theta_k\geq\frac{\sqrt{3}L_{3,\nu}}{(1+\nu)(2+\nu)}\alpha_k^{1+\nu}\|w_k\|_{x_k}^{*(1+\nu)}.
        \end{cases}
    \end{align*}
    Hence, \eqref{ineq:theta_p3} implies \eqref{ineq:theta_p3_alpha4}, \eqref{ineq:theta_p3_alpha2}, and \eqref{ineq:theta_p3_alpha2nu}.
\end{proof}

\subsection{Proof of Theorem~\ref{theorem:inexact_one_step}}\label{appendix:proof-inexact_one_step}
\begin{proof}[Proof of \cref{theorem:inexact_one_step}]
    If $w_k=0$, then \eqref{eq:scaled_newton_direction} gives $x_{k+1}=x_k$, and hence $r_{k+1}=\nabla f(x_{k+1})+s_k=w_k=0$.
    Thus, \eqref{ineq:one_step_decrease_inexact} holds trivially.
    We therefore assume that $w_k\neq0$.

    For the root Newton rule and the accuracy, \cref{lemma:q_r_relation_inexact} gives
    \begin{align}
        \theta_k=B_q\Delta_k^{\frac{q-2}{q-1}}=(9L_{p,\nu})^{\frac{1}{q-1}}(\kappa_\omega\Delta_k)^{\frac{q-2}{q-1}}\geq(9L_{p,\nu})^{\frac{1}{q-1}}\|w_k\|_{x_k}^{*\frac{q-2}{q-1}}, \label{ineq:theta_inexact_sufficient}
    \end{align}
    so \cref{proposition:residual_acceptance} gives \eqref{ineq:prod_res_dir_qk}.
    It is not necessary that earlier updates use the root stepsize.
    For an update accepted by \cref{alg:residual_backtracking}, \eqref{ineq:prod_res_dir_qk} follows directly from the acceptance condition \eqref{ineq:backtracking}.
    In both cases, $0<\theta_k\leq B_q\Delta_k^{\frac{q-2}{q-1}}$.
    Using the convexity of $f$, $\nabla f(x_{k+1})=r_{k+1}-s_k$, and $x_{k+1}-x_k=-\alpha_k\nabla^2f(x_k)^{-1}w_k$, we obtain
    \begin{align*}
        f(x_k)-f(x_{k+1})
        &\geq-\langle\nabla f(x_{k+1}),x_{k+1}-x_k\rangle\\
        &=\alpha_k\langle r_{k+1},\nabla^2f(x_k)^{-1}w_k\rangle-\langle s_k,x_k-x_{k+1}\rangle\\
        &\geqby{\eqref{ineq:fw_termination}}\alpha_k\langle r_{k+1},\nabla^2f(x_k)^{-1}w_k\rangle-\eta_k\\
        &\geqby{\eqref{ineq:prod_res_dir_qk}}\frac{c_p}{\theta_k}\|r_{k+1}\|_{x_k}^{*2}-\eta_k\\
        &\geq\frac{c_p}{B_q}\frac{\|r_{k+1}\|_{x_k}^{*2}}{\Delta_k^{\frac{q-2}{q-1}}}-\eta_k.
    \end{align*}
    This proves \eqref{ineq:one_step_decrease_inexact} for both choices.
\end{proof}

\subsection{Proof of Corollary~\ref{corollary:residual_backtracking}}\label{appendix:proof-residual_backtracking}
\begin{proof}[Proof of \cref{corollary:residual_backtracking}]
    Fix a nonterminal iteration $k$, so $\Delta_k>0$ and $\bar\theta_k>0$.
    Since $\nabla^2f(x_k)$ is a fixed positive definite matrix and $\X$ is compact, $D_k\coloneq\max_{u,v\in\X}\|u-v\|_{x_k}<\infty$.
    The proof of \cref{proposition:lmo_complexity} applies with $D$ replaced by $D_k$, so $\eta_k>0$ and $\alpha_{k,j}>0$ guarantee finite termination of each FW inner solve.
    For the away-step FW inner loop, fix any $\alpha\in(0,1]$ and $\eta>0$, and let $G_t$ be the FW gap of $\Psi_\alpha(\cdot;x_k)$.
    Since $k$ is nonterminal, $D_k>0$.
    By the direction-selection rule of \cref{alg:afw}, we have $\langle\nabla\Psi_\alpha(u_t;x_k),d_t\rangle\geq G_t$ and $\|d_t\|_{x_k}\leq D_k$.
    The quadratic expansion \eqref{eq:afw_quadratic_expansion} with $\alpha_k=\alpha$ and exact line minimization give
    \begin{align*}
    \Psi_\alpha(u_t;x_k)-\Psi_\alpha(u_{t+1};x_k)
    \geq
    \begin{cases}
        \dfrac{\alpha G_t^2}{2D_k^2}, & \gamma_t<\gamma_{t,\max},\\[2pt]
        \dfrac{G_t}{2}, & \text{FW step with }\gamma_t=1.
    \end{cases}
    \end{align*}
    Drop steps do not increase the model.
    Each drop step removes one point from the active set, whereas each FW step adds at most one point.
    Hence, the number of drop steps is at most the number of FW steps plus $|\Scal_0|-1$.
    Since $\Scal_0$ is finite, an infinite inner sequence would have infinitely many non-drop steps.
    If $G_t>\eta$ throughout, each such step would decrease the model by at least $\min\{\alpha\eta^2/(2D_k^2),\eta/2\}>0$, contradicting its lower boundedness on $\X$.
    Thus, the AFW inner loop terminates for every $\alpha\in(0,1]$ and $\eta>0$, in particular for $\alpha=\alpha_{k,j}$ and $\eta=\eta_k$.
    
    If no earlier trial is accepted, the search reaches a finite index $J$ with $\tau^J\bar\theta_k\geq B_q\Delta_k^{\frac{q-2}{q-1}}$.
    At this trial, $\theta_{k,J}=B_q\Delta_k^{\frac{q-2}{q-1}}$ and $\alpha_{k,J}=(1+B_q\Delta_k^{\frac{q-2}{q-1}})^{-1}$.
    If $w_{k,J}=0$, then $x_{k,J}=x_k$ and $r_{k+1,J}=0$, so \eqref{ineq:backtracking} holds trivially.
    Otherwise, $w_{k,J}\neq0$.
    Applying \cref{lemma:q_r_relation_inexact} to this trial gives $\|w_{k,J}\|_{x_k}^*\leq\kappa_\omega\Delta_k$, and hence
    \begin{align*}
        \theta_{k,J}=B_q\Delta_k^{\frac{q-2}{q-1}}=(9L_{p,\nu})^{\frac{1}{q-1}}(\kappa_\omega\Delta_k)^{\frac{q-2}{q-1}}\geq(9L_{p,\nu})^{\frac{1}{q-1}}\|w_{k,J}\|_{x_k}^{*\frac{q-2}{q-1}}.
    \end{align*}
    Thus, \cref{proposition:residual_acceptance} gives \eqref{ineq:backtracking}.
    When $p=2$ and $\nu=0$, we have $q=2$ and $\theta_{k,J}=B_2=9L_{2,0}$, so the same argument applies.
    In either case, backtracking terminates no later than trial $J$.
\end{proof}

\subsection{Proof of Proposition~\ref{proposition:potential_decrease}}\label{appendix:proof-potential_decrease}

\begin{lemma}[Uniform positive definiteness and comparison of local dual norms]\label[lemma]{lemma:uniform_hessian_bounds}
    Suppose that \cref{assumption:holder} holds with $q>2$.
    Then, there exist $\mu,L>0$ such that
    \begin{align}
        \mu I\preceq\nabla^2f(x)\preceq LI,\qquad x\in\X. \label{ineq:uniform_hessian_bounds}
    \end{align}
    Moreover, for any $x,y\in\X$ and $r\in\R^n$,
    \begin{align}
        \|r\|_x^{*2}\geq\frac{\mu}{L}\|r\|_y^{*2}. \label{ineq:grad_norm_ratio}
    \end{align}
\end{lemma}

\begin{proof}
    If $p=2$ and $\nu>0$, then \cref{assumption:holder} implies that $\nabla^2f$ is continuous, while $\nabla^2f$ is also continuous when $p=3$.
    Moreover, $\nabla^2f(x)\succ0$ holds by the assumption on $f$.
    Therefore, by the compactness of $\X$, there exist $\mu,L>0$ satisfying \eqref{ineq:uniform_hessian_bounds}.
    Hence,
    \begin{align*}
        \|r\|_x^{*2}=\langle\nabla^2f(x)^{-1}r,r\rangle\geq\frac{1}{L}\|r\|^2\geq\frac{\mu}{L}\langle\nabla^2f(y)^{-1}r,r\rangle=\frac{\mu}{L}\|r\|_y^{*2}.
    \end{align*}
\end{proof}

\begin{proof}[Proof of \cref{proposition:potential_decrease}]
    If the method terminates at initialization, there is no update to consider.
    Otherwise, $g_0>0$ implies $\Delta_0>0$, and $\Delta_k\geq\rho^k\Delta_0>0$ for every generated iterate.
    By \cref{theorem:inexact_one_step}, every RBNFW update satisfies \eqref{ineq:one_step_decrease_inexact}.
    By \eqref{ineq:one_step_decrease_inexact} and \eqref{ineq:grad_norm_ratio},
    \begin{align}
        f(x_k)-f(x_{k+1})
        &\geqby{\eqref{ineq:one_step_decrease_inexact}}\frac{c_p}{B_q}\frac{\|r_{k+1}\|_{x_k}^{*2}}{\Delta_k^{\frac{q-2}{q-1}}}-\eta_k
        \geqby{\eqref{ineq:grad_norm_ratio}}\frac{c_p\mu}{B_qL}\frac{\|r_{k+1}\|_{x_{k+1}}^{*2}}{\Delta_k^{\frac{q-2}{q-1}}}-\eta_k. \label{ineq:decrease_R_Delta}
    \end{align}
    Moreover, the accuracy \eqref{eq:global_inner_accuracy} gives
    \begin{align}
        \eta_k
        =\frac{\omega\Delta_k^2}{1+B_q\Delta_k^{\frac{q-2}{q-1}}}
        \leq\frac{\omega}{B_q}\Delta_k^{2-\frac{q-2}{q-1}}
        =\frac{\omega}{B_q}\Delta_k^{\frac{q}{q-1}}. \label{ineq:eta_Delta_power}
    \end{align}

    First, suppose that $\|r_{k+1}\|_{x_{k+1}}^*\geq\rho\Delta_k$.
    Then, $\Delta_{k+1}=\|r_{k+1}\|_{x_{k+1}}^*$.
    Since $\Delta_{k+1}\geq\rho\Delta_k$ and $\frac{q}{q-1}=2-\frac{q-2}{q-1}$, we have
    \begin{align*}
        \Delta_k^{\frac{q}{q-1}}
        =\frac{\Delta_k^2}{\Delta_k^{\frac{q-2}{q-1}}}
        \leq\rho^{-2}\frac{\Delta_{k+1}^2}{\Delta_k^{\frac{q-2}{q-1}}},
        \qquad
        \Delta_{k+1}^{\frac{q}{q-1}}
        =\frac{\Delta_{k+1}^2}{\Delta_{k+1}^{\frac{q-2}{q-1}}}
        \leq\rho^{-\frac{q-2}{q-1}}\frac{\Delta_{k+1}^2}{\Delta_k^{\frac{q-2}{q-1}}}.
    \end{align*}
    Hence, by \eqref{ineq:decrease_R_Delta} and \eqref{ineq:eta_Delta_power},
    \begin{align*}
        \Phi(x_k)-\Phi(x_{k+1})
        &=f(x_k)-f(x_{k+1})+\chi\left(\Delta_k^{\frac{q}{q-1}}-\Delta_{k+1}^{\frac{q}{q-1}}\right)\\
        &\geq\frac{c_p\mu}{B_qL}\frac{\Delta_{k+1}^2}{\Delta_k^{\frac{q-2}{q-1}}}-\frac{\omega}{B_q}\Delta_k^{\frac{q}{q-1}}-\chi\Delta_{k+1}^{\frac{q}{q-1}}\\
        &\geq\left(\frac{c_p\mu}{B_qL}-\frac{\omega}{B_q\rho^2}-\chi\rho^{-\frac{q-2}{q-1}}\right)\frac{\Delta_{k+1}^2}{\Delta_k^{\frac{q-2}{q-1}}}.
    \end{align*}

    Next, suppose that $\|r_{k+1}\|_{x_{k+1}}^*<\rho\Delta_k$.
    Then, $\Delta_{k+1}=\rho\Delta_k$.
    Moreover,
    \begin{align*}
        \Delta_k^{\frac{q}{q-1}}
        =\frac{1}{\rho^2}\frac{\rho^2\Delta_k^2}{\Delta_k^{\frac{q-2}{q-1}}}
        =\frac{1}{\rho^2}\frac{\Delta_{k+1}^2}{\Delta_k^{\frac{q-2}{q-1}}}.
    \end{align*}
    Therefore, by \eqref{ineq:decrease_R_Delta} and \eqref{ineq:eta_Delta_power},
    \begin{align*}
        \Phi(x_k)-\Phi(x_{k+1})
        &=f(x_k)-f(x_{k+1})+\chi\left(\Delta_k^{\frac{q}{q-1}}-\Delta_{k+1}^{\frac{q}{q-1}}\right)\\
        &\geq-\frac{\omega}{B_q}\Delta_k^{\frac{q}{q-1}}+\chi\left(1-\rho^{\frac{q}{q-1}}\right)\Delta_k^{\frac{q}{q-1}}\\
        &=\frac{\chi\left(1-\rho^{\frac{q}{q-1}}\right)-\omega/B_q}{\rho^2}\frac{\Delta_{k+1}^2}{\Delta_k^{\frac{q-2}{q-1}}}.
    \end{align*}

    Condition \eqref{ineq:omega_condition} is equivalent to the interval
    \begin{align}
        \frac{\omega}{B_q(1-\rho^{\frac{q}{q-1}})}
        <\chi<
        \frac{\rho^{\frac{q-2}{q-1}}}{B_q}\left(\frac{c_p\mu}{L}-\frac{\omega}{\rho^2}\right) \label{ineq:chi_interval}
    \end{align}
    being nonempty.
    Choose $\chi$ in this interval.
    Multiplying the upper bound in \eqref{ineq:chi_interval} by $\rho^{-\frac{q-2}{q-1}}$ and rearranging gives
    \begin{align*}
        \frac{c_p\mu}{B_qL}-\frac{\omega}{B_q\rho^2}-\chi\rho^{-\frac{q-2}{q-1}}>0.
    \end{align*}
    Rearranging the lower bound in \eqref{ineq:chi_interval} gives
    \begin{align*}
        \frac{\chi\left(1-\rho^{\frac{q}{q-1}}\right)-\omega/B_q}{\rho^2}>0.
    \end{align*}
    Hence, 
    \begin{align}
        c_\chi\coloneq\min\left\{\frac{c_p\mu}{B_qL}-\frac{\omega}{B_q\rho^2}-\chi\rho^{-\frac{q-2}{q-1}}, \frac{\chi(1-\rho^{\frac{q}{q-1}})-\omega/B_q}{\rho^2}\right\}>0, \label{eq:c_chi}
    \end{align}
    and \eqref{ineq:potential_decrease} holds in both cases.
\end{proof}

\subsection{Proof of Theorem~\ref{theorem:global_linear}}\label{appendix:proof-global_linear}

\begin{lemma}[Objective error bound in terms of the residual]\label[lemma]{lemma:objective_error_residual}
    Suppose that $f$ is $\mu$-strongly convex and $L$-smooth on $\X$.
    Let $x\in\X$, $s\in\R^n$, and $\eta\geq0$ satisfy
    \begin{align}
        \max_{v\in\X}\langle s,v-x\rangle\leq\eta, \label{ineq:eta_bound}
    \end{align}
    and let $r\coloneq\nabla f(x)+s$.
    Then,
    \begin{align*}
        f(x)-f(x^\star)\leq\frac{L}{2\mu}\|r\|_x^{*2}+\eta.
    \end{align*}
\end{lemma}

\begin{proof}
    By the $\mu$-strong convexity of $f$,
    \begin{align*}
        f(x)-f(x^\star)
        &\leq\langle\nabla f(x),x-x^\star\rangle-\frac{\mu}{2}\|x-x^\star\|^2\\
        &=\langle r,x-x^\star\rangle+\langle s,x^\star-x\rangle-\frac{\mu}{2}\|x-x^\star\|^2\\
        &\leqby{(a)}\sqrt{L}\|r\|_x^*\|x-x^\star\|+\eta-\frac{\mu}{2}\|x-x^\star\|^2\\
        &\leqby{(b)}\frac{L}{2\mu}\|r\|_x^{*2}+\eta.
    \end{align*}
    In (a), we used \eqref{ineq:eta_bound}, the Cauchy--Schwarz inequality, and $\nabla^2f(x)\preceq LI$ to obtain
    \begin{align*}
        \langle r,x-x^\star\rangle\leq\|r\|_x^*\|x-x^\star\|_x\leq\sqrt{L}\|r\|_x^*\|x-x^\star\|.
    \end{align*}
    In (b), we used Young's inequality
    \begin{align*}
        \sqrt{L}\|r\|_x^*\|x-x^\star\|\leq\frac{L}{2\mu}\|r\|_x^{*2}+\frac{\mu}{2}\|x-x^\star\|^2.
    \end{align*}
\end{proof}

\begin{proof}[Proof of \cref{theorem:global_linear}]
    If $g_0=0$, the method terminates immediately and the bound at $k=0$ is trivial.
    Otherwise, $\Delta_0>0$, and \cref{corollary:residual_backtracking} guarantees that every outer update is well defined.
    Let $i\geq0$ be an index for which the iterate is generated.
    For $i=0$, set $s_{-1}=0$ and $\eta_{-1}=0$, so \eqref{ineq:eta_bound} holds immediately.
    For $i\geq1$, \eqref{ineq:fw_termination} at iteration $i-1$ yields \eqref{ineq:eta_bound} with $x=x_i$, $s=s_{i-1}$, and $\eta=\eta_{i-1}$.
    Therefore, applying \cref{lemma:objective_error_residual} with $r=r_i$ gives
    \begin{align*}
        f(x_i)-f(x^\star)\leq\frac{L}{2\mu}\|r_i\|_{x_i}^{*2}+\eta_{i-1}.
    \end{align*}
    For $i\geq1$, the accuracy and $\Delta_i\geq\rho\Delta_{i-1}$ imply
    \begin{align*}
        \eta_{i-1}=\frac{\omega\Delta_{i-1}^2}{1+B_q\Delta_{i-1}^{\frac{q-2}{q-1}}}\leq\omega\Delta_{i-1}^2\leq\omega\rho^{-2}\Delta_i^2.
    \end{align*}
    The same bound holds for $i=0$ because $\eta_{-1}=0$.
    Therefore,
    \begin{align}
        f(x_i)-f(x^\star)\leq\left(\frac{L}{2\mu}+\omega\rho^{-2}\right)\Delta_i^2. \label{ineq:error_bound_delta}
    \end{align}

    We next use \eqref{ineq:potential_decrease}.
    Since $\chi\Delta_i^{\frac{q}{q-1}}\leq\Phi(x_i)\leq\Phi(x_0)$, we have $\Delta_i\leq\left(\frac{\Phi(x_0)}{\chi}\right)^{\frac{q-1}{q}}$.
    Hence, \eqref{ineq:error_bound_delta} yields
    \begin{align*}
        f(x_i)-f(x^\star)
        \leq\left(\frac{L}{2\mu}+\omega\rho^{-2}\right)\Delta_i^2
        \leq\left(\frac{L}{2\mu}+\omega\rho^{-2}\right)\left(\frac{\Phi(x_0)}{\chi}\right)^{\frac{q-2}{q}}\Delta_i^{\frac{q}{q-1}}.
    \end{align*}
    Therefore,
    \begin{align}
        \Phi(x_i)
        =f(x_i)-f(x^\star)+\chi\Delta_i^{\frac{q}{q-1}}
        \leq\left(\chi+\left(\frac{L}{2\mu}+\omega\rho^{-2}\right)\left(\frac{\Phi(x_0)}{\chi}\right)^{\frac{q-2}{q}}\right)\Delta_i^{\frac{q}{q-1}}. \label{ineq:phi_delta_power}
    \end{align}

    Now let $i$ be a nonterminal index, so $x_{i+1}$ is generated and $\Delta_i,\Delta_{i+1}>0$.
    Since $\Delta_{i+1}\geq\rho\Delta_i$ and $\frac{q}{q-1}=2-\frac{q-2}{q-1}$,
    \begin{align*}
        \frac{\Delta_{i+1}^2}{\Delta_i^{\frac{q-2}{q-1}}}
        =\Delta_{i+1}^{2-\frac{q-2}{q-1}}\left(\frac{\Delta_{i+1}}{\Delta_i}\right)^{\frac{q-2}{q-1}}
        =\Delta_{i+1}^{\frac{q}{q-1}}\left(\frac{\Delta_{i+1}}{\Delta_i}\right)^{\frac{q-2}{q-1}}
        \geq\rho^{\frac{q-2}{q-1}}\Delta_{i+1}^{\frac{q}{q-1}}.
    \end{align*}
    Moreover, \eqref{ineq:phi_delta_power} gives
    \begin{align*}
        \Delta_{i+1}^{\frac{q}{q-1}}
        \geq\frac{\Phi(x_{i+1})}{\chi+\left(\frac{L}{2\mu}+\omega\rho^{-2}\right)\left(\frac{\Phi(x_0)}{\chi}\right)^{\frac{q-2}{q}}}.
    \end{align*}
    Consequently, \eqref{ineq:potential_decrease} implies
    \begin{align*}
        \Phi(x_i)-\Phi(x_{i+1})
        \geq\frac{c_\chi\rho^{\frac{q-2}{q-1}}}{\chi+\left(\frac{L}{2\mu}+\omega\rho^{-2}\right)\left(\frac{\Phi(x_0)}{\chi}\right)^{\frac{q-2}{q}}}\Phi(x_{i+1}).
    \end{align*}
    Equivalently,
    \begin{align*}
        \Phi(x_{i+1})
        \leq\left(1+\frac{c_\chi\rho^{\frac{q-2}{q-1}}}{\chi+\left(\frac{L}{2\mu}+\omega\rho^{-2}\right)\left(\frac{\Phi(x_0)}{\chi}\right)^{\frac{q-2}{q}}}\right)^{-1}\Phi(x_i).
    \end{align*}
    Iterating this inequality for $i=0,1,\dots,k-1$ yields the upper bound in \eqref{ineq:global_linear_obj}, while the lower bound follows from $f(x_k)-f(x^\star)\leq\Phi(x_k)$.
\end{proof}

\section{Proofs of Section~\ref{subsec:local-convergence}}\label{appendix:proofs-local-convergence}
After a finite number of iterations, we switch to the full-step regime by setting $\alpha_k=1$ and solving each subproblem until the stopping criterion in~\eqref{eq:model_decrease_stopping} is satisfied.
The local analysis applies to RBNFW after switching.

\subsection{Proof of Lemma~\ref{lemma:finite_inner_iterations_error_estimate}}\label{appendix:proof-finite_inner_iterations_error_estimate}
\begin{proof}[Proof of~\cref{lemma:finite_inner_iterations_error_estimate}]
    Let $M_k^*\coloneq f(x_k)-\Psi_{1}(\widehat x_{k+1};x_k)$.
    Since $g_k>0$, there exists $v\in\X$ such that $\langle\nabla f(x_k),v-x_k\rangle<0$.
    Hence, for all $0<\gamma<\min\left\{1,\frac{-2\langle\nabla f(x_k),v-x_k\rangle}{\|v-x_k\|_{x_k}^2}\right\}$, the point $x_k+\gamma(v-x_k)$ is feasible and
    \begin{align*}
        \Psi_{1}(x_k+\gamma(v-x_k);x_k)-f(x_k)
        =\gamma\langle\nabla f(x_k),v-x_k\rangle+\frac{\gamma^2}{2}\|v-x_k\|_{x_k}^2<0,
    \end{align*}
    and therefore $M_k^*>0$ due to $\Psi_1(\widehat x_{k+1};x_k) \leq \Psi_1(x_k + \gamma(v-x_k);x_k) < f(x_k)$.

    Since the inner updates use exact line minimization,
    \begin{align*}
        \Psi_{1}(u_{t+1};x_k)\leq\Psi_{1}(u_t;x_k)\leq\Psi_{1}(u_0;x_k)=f(x_k),
    \end{align*}
    and hence $0\leq M_t\leq M_k^*$.
    The convexity of $\Psi_1(\cdot;x_k)$ gives
    \begin{align*}
        0\leq M_k^*-M_t
        =\Psi_{1}(u_t;x_k)-\Psi_{1}(\widehat x_{k+1};x_k)
        \leq G_t.
    \end{align*}
    Fix any $a>0$ and suppose, for contradiction, that $\ell_k(a)$ is infinite, so the inner updates continue indefinitely.
    For FW, \cref{proposition:lmo_complexity} guarantees finite termination for every positive FW-gap tolerance.
    For away-step FW, the same follows from the inner-loop argument in the proof of \cref{corollary:residual_backtracking} with $\alpha=1$.
    Since the update rules do not depend on the stopping tolerance, applying these arguments to the same inner sequence with positive tolerances tending to zero yields inner indices $\{t_j\}$ such that $G_{t_j}\to0$.
    Consequently, $M_{t_j}\to M_k^*>0$, and, for any $a>0$,
    \begin{align*}
        \left(\frac{M_{t_j}}{\|\nabla f(x_k)\|}\right)^a\to\left(\frac{M_k^*}{\|\nabla f(x_k)\|}\right)^a>0.
    \end{align*}
    Therefore, the condition in \eqref{eq:model_decrease_stopping} holds for all sufficiently large $j$, a contradiction.
    Hence, $\ell_k(a)$ is finite.
    Furthermore, by the $\mu$-strong convexity of $\Psi_{1}(\cdot;x_k)$ and \eqref{eq:model_decrease_stopping},
    \begin{align*}
        \frac{\mu}{2}\|u_{\ell_k(a)}-\widehat x_{k+1}\|^2
        \leq G_{\ell_k(a)}
        \leq\left(\frac{M_{\ell_k(a)}}{\|\nabla f(x_k)\|}\right)^a.
    \end{align*}
    On the other hand,
    \begin{align*}
        M_{\ell_k(a)}
        \leq M_k^*
        =-\langle\nabla f(x_k),\widehat x_{k+1}-x_k\rangle-\frac12\|\widehat x_{k+1}-x_k\|_{x_k}^2
        \leq\|\nabla f(x_k)\|\|\widehat x_{k+1}-x_k\|,
    \end{align*}
    where the last inequality uses the Cauchy--Schwarz inequality and $-\frac12\|\widehat x_{k+1}-x_k\|_{x_k}^2<0$.
    Combining the last two displays yields \eqref{ineq:model_decrease_gap_bound} and \eqref{ineq:model_decrease_inner_error}.
\end{proof}

\subsection{Proof of Theorem~\ref{theorem:local_convergence_p2}}\label{appendix:proof-local_convergence_p2}
\begin{lemma}[Vanishing of the FW gap]\label[lemma]{lemma:fw_gap_convergence}
    Let $\{x_k\}\subset\X$ satisfy $f(x_k)-f(x^\star)\to0$.
    Then, $g_k\to0$.
\end{lemma}

\begin{proof}
    By the strong convexity of $f$, $\frac{\mu}{2}\|x_k-x^\star\|^2\leq f(x_k)-f(x^\star)\to0$.
    Hence, $x_k\to x^\star$.
    Since $\X$ is compact and $\nabla f$ is continuous, the FW-gap function $x\mapsto\max_{v\in\X}\langle\nabla f(x),x-v\rangle$ is continuous.
    Moreover, the optimality of $x^\star$ implies that its FW gap is zero.
    Therefore, $g_k\to0$.
\end{proof}

\begin{proof}[Proof of \cref{theorem:local_convergence_p2}]
    If the switching criterion in \eqref{eq:p2_model_decrease_switch} were never satisfied, RBNFW with \eqref{eq:global_inner_accuracy} would be used at every iteration.
    Then, \cref{theorem:global_linear,lemma:fw_gap_convergence} would give $g_k\to0$, contradicting the failure of the switching criterion.
    Therefore, $K_2$ is finite.

    If $g_j=0$ for some $j\geq K_2$, then the strong convexity of $f$ gives $x_j=x^\star$, so the sequence reaches the solution in finitely many iterations.
    Hence, it remains to consider the case $g_k>0$ for every $k\geq K_2$.
    Let $k\geq K_2$, define $e_k\coloneq x_k-x^\star$, and let $\widehat x_{k+1}$ be the exact minimizer of $\Psi_1(\cdot;x_k)$ over $\X$.
    By the optimality conditions of $\widehat x_{k+1}$ and $x^\star$,
    \begin{align*}
        \|\widehat x_{k+1}-x^\star\|_{x_k}^2\leq\left\langle\nabla^2f(x_k)e_k-(\nabla f(x_k)-\nabla f(x^\star)),\widehat x_{k+1}-x^\star\right\rangle.
    \end{align*}
    By the Cauchy--Schwarz inequality and the H\"older continuity of the Hessian,
    \begin{align*}
        \|\widehat x_{k+1}-x^\star\|_{x_k}
        &\leq\left\|\int_0^1\left(\nabla^2f(x_k)-\nabla^2f(x^\star+te_k)\right)e_k\,dt\right\|_{x_k}^*\\
        &\leq L_{2,\nu}\|e_k\|_{x_k}^{1+\nu}\int_0^1(1-t)^\nu\,dt
        =\frac{L_{2,\nu}}{1+\nu}\|e_k\|_{x_k}^{1+\nu}.
    \end{align*}
    Therefore,
    \begin{align}
        \|\widehat x_{k+1}-x^\star\|\leq\frac{L_{2,\nu}L^{\frac{1+\nu}{2}}}{(1+\nu)\sqrt{\mu}}\|e_k\|^{1+\nu}. \label{ineq:p2_model_decrease_exact}
    \end{align}

    Using \eqref{ineq:p2_model_decrease_exact}, $\|e_k\|\leq D/\sqrt{\mu}$, and the triangle inequality,
    \begin{align*}
        \|\widehat x_{k+1}-x_k\|
        \leq\|\widehat x_{k+1}-x^\star\|+\|x^\star-x_k\|
        \leq\left(1+\frac{L_{2,\nu}L^{\frac{1+\nu}{2}}D^\nu}{(1+\nu)\mu^{\frac{1+\nu}{2}}}\right)\|e_k\|.
    \end{align*}
    Combining this inequality with \eqref{ineq:model_decrease_inner_error} with $a=2(1+\nu)$ gives
    \begin{align*}
        \|x_{k+1}-\widehat x_{k+1}\|
        \leq\sqrt{\frac{2}{\mu}}\left(1+\frac{L_{2,\nu}L^{\frac{1+\nu}{2}}D^\nu}{(1+\nu)\mu^{\frac{1+\nu}{2}}}\right)^{1+\nu}\|e_k\|^{1+\nu}.
    \end{align*}
    Hence, by \eqref{ineq:p2_model_decrease_exact} and the triangle inequality,
    \begin{align*}
        \|x_{k+1}-x^\star\|
        \leq\|x_{k+1}-\widehat x_{k+1}\|+\|\widehat x_{k+1}-x^\star\|
        \leq C_2\|e_k\|^{1+\nu}.
    \end{align*}
    This proves \eqref{ineq:p2_model_decrease_local_order}.

    By \eqref{eq:p2_model_decrease_switch}, strong convexity, and convexity,
    \begin{align*}
        \frac{\mu}{2}\|e_{K_2}\|^2\leq f(x_{K_2})-f(x^\star)\leq g_{K_2}\leq\frac{\mu}{2}(2C_2)^{-\frac{2}{\nu}}.
    \end{align*}
    Therefore, the switching condition implies $C_2\|e_{K_2}\|^\nu\leq\frac12$.
    Moreover, if $C_2\|e_k\|^\nu\leq1/2$, then
    \begin{align*}
        \|e_{k+1}\|\leq C_2\|e_k\|^{1+\nu}\leq\frac12\|e_k\|,
        \qquad
        C_2\|e_{k+1}\|^\nu\leq2^{-\nu}C_2\|e_k\|^\nu\leq\frac12.
    \end{align*}
    Thus, by induction, $x_k\to x^\star$ and $\|e_{k+1}\|\leq\frac12\|e_k\|$ for every $k\geq K_2$.

    Since $\alpha_k=1$, the definition of $s_k$, \cref{lemma:holder}, $\nabla^2f(x_k)\preceq LI$, and the preceding contraction give
    \begin{align*}
        \|\nabla f(x_{k+1})+s_k\|_{x_k}^*
        \leq\frac{L_{2,\nu}}{1+\nu}\|x_{k+1}-x_k\|_{x_k}^{1+\nu}
        \leq\frac{L_{2,\nu}}{1+\nu}\left(\frac{3\sqrt{L}}{2}\right)^{1+\nu}\|e_k\|^{1+\nu}.
    \end{align*}
    Moreover, \eqref{ineq:model_decrease_gap_bound} and the bound on $\|\widehat x_{k+1}-x_k\|$ established above imply
    \begin{align*}
        G_{\ell_k(2(1+\nu))}
        \leq\left(1+\frac{L_{2,\nu}L^{(1+\nu)/2}D^\nu}{(1+\nu)\mu^{(1+\nu)/2}}\right)^{2(1+\nu)}\|e_k\|^{2(1+\nu)}.
    \end{align*}
    Since $\nabla\Psi_1(x_{k+1};x_k)=-s_k$, we have $\max_{v\in\X}\langle s_k,v-x_{k+1}\rangle=G_{\ell_k(2(1+\nu))}$.
    Applying \cref{lemma:objective_error_residual} with $x=x_{k+1}$, $s=s_k$, and $\eta=G_{\ell_k(2(1+\nu))}$, and using \eqref{ineq:grad_norm_ratio} in the form $\|\nabla f(x_{k+1})+s_k\|_{x_{k+1}}^{*2}\leq\frac{L}{\mu}\|\nabla f(x_{k+1})+s_k\|_{x_k}^{*2}$ together with $\frac{\mu}{2}\|e_k\|^2\leq f(x_k)-f(x^\star)$, we obtain, for some constant $\widetilde C_2>0$ independent of $k$,
    \begin{align*}
        f(x_{k+1})-f(x^\star)
        \leq\frac{L}{2\mu}\|\nabla f(x_{k+1})+s_k\|_{x_{k+1}}^{*2}+G_{\ell_k(2(1+\nu))}
        \leq\widetilde C_2\left(f(x_k)-f(x^\star)\right)^{1+\nu}.
    \end{align*}
    This proves \eqref{ineq:p2_model_decrease_local_objective} and completes the proof.
\end{proof}

\subsection{Proof of Theorem~\ref{theorem:local_convergence_p3}}\label{appendix:proof-local_convergence_p3}
\begin{proof}[Proof of \cref{theorem:local_convergence_p3}]
    If the switching criterion in \eqref{eq:p3_quadratic_switch} were never satisfied, RBNFW with \eqref{eq:global_inner_accuracy} would be used at every iteration.
    Then, \cref{theorem:global_linear,lemma:fw_gap_convergence} would give $g_k\to0$, contradicting the failure of the switching criterion.
    Therefore, $K_3$ is finite.

    If $g_j=0$ for some $j\geq K_3$, then the strong convexity of $f$ gives $x_j=x^\star$, so the sequence reaches the solution in finitely many iterations.
    Hence, it remains to consider the case $g_k>0$ for every $k\geq K_3$.
    Let $k\geq K_3$, define $e_k\coloneq x_k-x^\star$, and let $\widehat x_{k+1}$ be the exact minimizer of $\Psi_1(\cdot;x_k)$ over $\X$.
    By the optimality conditions of $\widehat x_{k+1}$ and $x^\star$,
    \begin{align*}
        \left\langle\nabla f(x_k)+\nabla^2f(x_k)(\widehat x_{k+1}-x_k),x^\star-\widehat x_{k+1}\right\rangle\geq0,
        \qquad
        \langle\nabla f(x^\star),\widehat x_{k+1}-x^\star\rangle\geq0.
    \end{align*}
    Therefore,
    \begin{align*}
        \|\widehat x_{k+1}-x^\star\|_{x_k}^2
        \leq\left\langle\nabla^2f(x_k)e_k-(\nabla f(x_k)-\nabla f(x^\star)),\widehat x_{k+1}-x^\star\right\rangle.
    \end{align*}
    By the Cauchy--Schwarz inequality,
    \begin{align*}
        \|\widehat x_{k+1}-x^\star\|_{x_k}
        \leq\|\nabla^2f(x_k)e_k-(\nabla f(x_k)-\nabla f(x^\star))\|_{x_k}^*.
    \end{align*}
    Using $\mu I\preceq\nabla^2f(x_k)$ and \eqref{ineq:p3_hessian_lipschitz}, we obtain
    \begin{align}
        \|\widehat x_{k+1}-x^\star\|
        &\leq\frac{1}{\sqrt{\mu}}\|\widehat x_{k+1}-x^\star\|_{x_k}\notag\\
        &\leq\frac{1}{\sqrt{\mu}}\|\nabla^2f(x_k)e_k-(\nabla f(x_k)-\nabla f(x^\star))\|_{x_k}^*\notag\\
        &\leq\frac{1}{\mu}\left\|\nabla^2f(x_k)e_k-(\nabla f(x_k)-\nabla f(x^\star))\right\|\notag\\
        &=\frac{1}{\mu}\left\|\int_0^1\left(\nabla^2f(x_k)-\nabla^2f(x^\star+te_k)\right)e_k\,dt\right\|\notag\\
        &\leq\frac{1}{\mu}\int_0^1\|\nabla^2f(x_k)-\nabla^2f(x^\star+te_k)\|_2\|e_k\|\,dt\notag\\
        &\leq\frac{M}{\mu}\|e_k\|^2\int_0^1(1-t)\,dt
        =\frac{M}{2\mu}\|e_k\|^2. \label{ineq:p3_exact_quadratic}
    \end{align}

    Using \eqref{ineq:p3_exact_quadratic}, $\|e_k\|\leq D/\sqrt{\mu}$, and the triangle inequality,
    \begin{align*}
        \|\widehat x_{k+1}-x_k\|
        \leq\|e_k\|+\|\widehat x_{k+1}-x^\star\|
        \leq\left(1+\frac{M}{2\mu}\|e_k\|\right)\|e_k\|
        \leq\left(1+\frac{MD}{2\mu^{3/2}}\right)\|e_k\|.
    \end{align*}
    Combining this inequality with \eqref{ineq:model_decrease_inner_error} with $a=4$ gives $\|x_{k+1}-\widehat x_{k+1}\|\leq\sqrt{\frac{2}{\mu}}\left(1+\frac{MD}{2\mu^{3/2}}\right)^2\|e_k\|^2$.
    Hence, by \eqref{ineq:p3_exact_quadratic} and the triangle inequality,
    \begin{align*}
        \|x_{k+1}-x^\star\|
        \leq\|x_{k+1}-\widehat x_{k+1}\|+\|\widehat x_{k+1}-x^\star\|
        \leq C_3\|e_k\|^2.
    \end{align*}
    This proves \eqref{ineq:p3_local_quadratic}.

    By \eqref{eq:p3_quadratic_switch}, strong convexity, and convexity,
    \begin{align*}
        \frac{\mu}{2}\|e_{K_3}\|^2\leq f(x_{K_3})-f(x^\star)\leq g_{K_3}\leq\frac{\mu}{8C_3^2}.
    \end{align*}
    Therefore, $C_3\|e_{K_3}\|\leq\frac12$.
    If the same inequality holds at iteration $k$, then \eqref{ineq:p3_local_quadratic} gives $\|e_{k+1}\|\leq\frac12\|e_k\|$ and $C_3\|e_{k+1}\|\leq\frac14$.
    Thus, by induction, $x_k\to x^\star$ and $\|e_{k+1}\|\leq\frac12\|e_k\|$ for every $k\geq K_3$.

    Since $\alpha_k=1$, the definition of $s_k$, \eqref{ineq:p3_hessian_lipschitz}, and the preceding contraction give
    \begin{align*}
        \|\nabla f(x_{k+1})+s_k\|_{x_{k+1}}^*
        &\leq\frac{1}{\sqrt{\mu}}\|\nabla f(x_{k+1})-\nabla f(x_k)-\nabla^2f(x_k)(x_{k+1}-x_k)\|\\
        &\leq\frac{M}{2\sqrt{\mu}}\|x_{k+1}-x_k\|^2\leq\frac{9M}{8\sqrt{\mu}}\|e_k\|^2.
    \end{align*}
    Moreover, \eqref{ineq:model_decrease_gap_bound} and the bound on $\|\widehat x_{k+1}-x_k\|$ established above imply $G_{\ell_k(4)}\leq\left(1+\frac{MD}{2\mu^{3/2}}\right)^4\|e_k\|^4$.
    Since $\nabla\Psi_1(x_{k+1};x_k)=-s_k$, we have $\max_{v\in\X}\langle s_k,v-x_{k+1}\rangle=G_{\ell_k(4)}$.
    Applying \cref{lemma:objective_error_residual} with $x=x_{k+1}$, $s=s_k$, and $\eta=G_{\ell_k(4)}$, and using the preceding bounds together with $\frac{\mu}{2}\|e_k\|^2\leq f(x_k)-f(x^\star)$, we obtain, for some constant $\widetilde C_3>0$ independent of $k$,
    \begin{align*}
        f(x_{k+1})-f(x^\star)
        \leq\frac{L}{2\mu}\|\nabla f(x_{k+1})+s_k\|_{x_{k+1}}^{*2}+G_{\ell_k(4)}
        \leq\widetilde C_3\left(f(x_k)-f(x^\star)\right)^2.
    \end{align*}
    This proves \eqref{ineq:p3_local_quadratic_objective} and completes the proof.
\end{proof}

\subsection{Total LMO Complexity}\label{appendix:total_lmo_complexity}
Finally, we summarize the total LMO complexity of our methods and several related methods.
For the global bounds, we use the global stepsize and inner accuracy rules throughout under the assumptions of \cref{theorem:global_linear}.
By \eqref{eq:global_inner_accuracy}, \eqref{ineq:error_bound_delta}, and $\Phi(x_k)\leq\Phi(x_0)$, we have $\eta_k\geq c\epsilon$ and $\alpha_{k,j}\geq\underline\alpha$ for all trials whenever $f(x_k)-f(x^\star)>\epsilon$, with $c,\underline\alpha>0$ independent of $\epsilon$.
Since $g_k$ is bounded and $|\widehat{\Scal}_k|\leq|\vertex\X|<\infty$ in the polytope case, \cref{proposition:lmo_complexity,proposition:lmo_complexity_away} give $\O(\epsilon^{-1})$ and $\O(\log\epsilon^{-1})$ LMO calls per inner solve, respectively.
For the accepted trial index $j_k$ in \cref{alg:residual_backtracking}, acceptance at the root value, $\Delta_k\geq\rho\Delta_{k-1}$, and the update $\bar\theta_{k+1}=\theta_k/\tau$ give $j_k\leq 2+\frac{q-2}{q-1}\log_\tau(\rho^{-1})+\log_\tau(\bar\theta_{k+1}/\bar\theta_k)$ for $k\geq1$.
Since $\bar\theta_k$ is bounded above and $j_0$ is bounded independently of $\epsilon$, telescoping gives $\O(1)$ inner solves per outer iteration on average; the root Newton rule uses one.
Combining these estimates with the $\O(\log\epsilon^{-1})$ outer iteration bound of \cref{theorem:global_linear} yields global total LMO complexities $\O(\epsilon^{-1}\log\epsilon^{-1})$ and $\O((\log\epsilon^{-1})^2)$ for \cref{alg:dnfw,alg:dnfw_away}, respectively.
These bounds hold with problem and algorithmic parameters fixed and include initialization, outer FW-gap evaluations, and rejected trials.

For the local bounds, we set $a=2(1+\nu)$ for $p=2$ and $a=4$ for $p=3$.
Writing $M_k^*\coloneq f(x_k)-\Psi_1(\widehat x_{k+1};x_k)$, comparison with the feasible point $x_k+(\mu/L)(x^\star-x_k)$ gives $M_k^*\geq(\mu/L)(f(x_k)-f(x^\star))$ by strong convexity and $\nabla^2f(x_k)\preceq LI$.
Define the analysis-only tolerance $\bar\eta_k\coloneq\min\left\{\frac{M_k^*}{2},\left(\frac{M_k^*}{2\|\nabla f(x_k)\|}\right)^a\right\}$.
Since $M_k^*-M_t\leq G_t$, the inequality $G_t\leq\bar\eta_k$ implies $M_t\geq M_k^*/2$ and hence the stopping condition \eqref{eq:model_decrease_stopping}.
Thus, the local inner loop terminates no later than the same inner method run with fixed tolerance $\bar\eta_k$.
Since $\nabla f$ is bounded on $\X$, there exists $c>0$, independent of $k$, such that $\bar\eta_k\geq c(f(x_k)-f(x^\star))^a$ whenever $0<f(x_k)-f(x^\star)\leq1$.
Thus, for $0<f(x_k)-f(x^\star)\leq 1/e$, \cref{proposition:lmo_complexity} gives $\O((f(x_k)-f(x^\star))^{-a})$ LMO calls per FW inner solve, while the contraction estimate in the proof of \cref{proposition:lmo_complexity_away} gives $\O(\log (f(x_k)-f(x^\star))^{-1})$ non-drop steps per AFW inner solve.
Finite termination has an $\epsilon$-independent cost, so assume that $f(x_k)-f(x^\star)>0$ for all $k$.
By \cref{theorem:local_convergence_p2,theorem:local_convergence_p3}, $f(x_k)-f(x^\star)\to0$ and $f(x_{k+1})-f(x^\star)\leq\widetilde C_p (f(x_k)-f(x^\star))^{a/2}$ after switching.
Fix $l\in(1,a/2)$.
There exists $k_0$ after switching, independent of $\epsilon$, such that $0<f(x_k)-f(x^\star)\leq 1/e$, $f(x_{k+1})-f(x^\star)\leq (f(x_k)-f(x^\star))/2$, and $\log (f(x_{k+1})-f(x^\star))^{-1}\geq l\log (f(x_k)-f(x^\star))^{-1}$ for all $k\geq k_0$.
For sufficiently small $\epsilon>0$, let $N\coloneq\min\{k\geq k_0\mid f(x_k)-f(x^\star)\leq\epsilon\}$.
Since $f(x_{N-1})-f(x^\star)>\epsilon$, geometric summation yields
\begin{align}
    \sum_{k=k_0}^{N-1}(f(x_k)-f(x^\star))^{-a}
    &\leq\frac{(f(x_{N-1})-f(x^\star))^{-a}}{1-2^{-a}}
    =\O(\epsilon^{-a}), \notag\\
    \sum_{k=k_0}^{N-1}\log (f(x_k)-f(x^\star))^{-1}
    &\leq\frac{l}{l-1}\log (f(x_{N-1})-f(x^\star))^{-1}
    =\O(\log\epsilon^{-1}). \notag
\end{align}
Since active sets are carried between solves, the total number of drop steps is at most the number of FW steps plus the active-set size at switching.
Hence, the local total LMO complexities, including outer FW-gap evaluations, are $\O(\epsilon^{-a})$ and $\O(\log\epsilon^{-1})$ for \cref{alg:dnfw,alg:dnfw_away}, respectively.
Problem and algorithmic parameters are fixed, including $\nu>0$ when $p=2$.
The finite LMO cost before $k_0$, including rejected trials, is independent of $\epsilon$ and absorbed in these bounds.

These bounds and those for related methods are summarized in \cref{tab:total_lmo_complexity}.
The local bounds are conservative worst-case guarantees based on sufficient inner accuracy conditions for fast local convergence.
In the experiments of \cref{sec:numerical-experiments}, our methods were competitive in LMO calls and often achieved the best wall-clock performance among the tested methods.

\begin{table}[!tbp]
    \caption{Comparison of global and local total LMO complexity for attaining $f(x_k)-f(x^\star)\leq\epsilon$. All inner LMO calls are included. For the proposed methods, global bounds use the global accuracy and stepsize rules throughout, including rejected trials, whereas local bounds count calls after switching. Our $\C^2$ bounds require $\nu>0$. Here, $\nu_{\mathrm N}\in(1,1.139)$. cvx: convex, scvx: strongly convex, self-conc.: self-concordant.}
    \label{tab:total_lmo_complexity}
    \centering
    \resizebox{\textwidth}{!}{%
    \begin{tabular}{cccccll}
        \toprule
        Algorithm & $f$ & $\nabla f$ Lip. & $\nabla^p f$ & $\X$ & global & local \\ \midrule
        FW & $\C^1$ cvx & \cmarkg & \xmarkr & cvx
        & $\O(\epsilon^{-1})$
        & --- \\
        CGS & $\C^1$ cvx & \cmarkg & \xmarkr & cvx
        & $\O(\epsilon^{-1})$
        & --- \\
        r-CGS & $\C^1$ scvx & \cmarkg & \xmarkr & cvx
        & $\O(\epsilon^{-1})$
        & --- \\
        NFW & $\C^3$ cvx & \xmarkr & self-conc. & cvx
        & ---
        & $\O(\epsilon^{-\nu_{\mathrm N}})$ \\
        SOCGS & $\C^2$ scvx & \cmarkg & $\nabla^2 f$ Lip. & polytope
        & $\O((\log\epsilon^{-1})^2)$
        & $\O(\log\epsilon^{-1}\log\log\epsilon^{-1})$ \\ \midrule
        \textbf{Alg.~\ref{alg:dnfw}} & $\C^2$ scvx & \cmarkg & $\nabla^2 f$ H\"older & cvx
        & $\O(\epsilon^{-1}\log\epsilon^{-1})$
        & $\O(\epsilon^{-2(1+\nu)})$ \\
        \textbf{Alg.~\ref{alg:dnfw_away}} & $\C^2$ scvx & \cmarkg & $\nabla^2 f$ H\"older & polytope
        & $\O((\log\epsilon^{-1})^2)$
        & $\O(\log\epsilon^{-1})$ \\
        \textbf{Alg.~\ref{alg:dnfw}} & $\C^3$ scvx & \cmarkg & $\nabla^3 f$ H\"older & cvx
        & $\O(\epsilon^{-1}\log\epsilon^{-1})$
        & $\O(\epsilon^{-4})$ \\
        \textbf{Alg.~\ref{alg:dnfw_away}} & $\C^3$ scvx & \cmarkg & $\nabla^3 f$ H\"older & polytope
        & $\O((\log\epsilon^{-1})^2)$
        & $\O(\log\epsilon^{-1})$ \\
        \bottomrule
    \end{tabular}}
\end{table}

\section{Additional Experimental Results}\label{appendix:additional_experiments}
\subsection{Matrix Sensing with Squared Loss}\label{appendix:matrix_sensing}

We generated the ground truth matrix $X_\star\in\M_{R_{\tr}}$, and formed the deterministic observations $(y_{\mathrm{obs}})_i=\langle A_i, X_\star \rangle$, $i=1,\ldots,m$.
We chose orthonormal matrices $V_1,\ldots,V_\varphi\in\S^n$ satisfying $\tr(V_j)=0$, and extended them to an orthonormal basis of $\S^n$ with $W_1,\ldots,W_{m-\varphi}$ chosen orthogonal to every $V_j$.
For the condition number $\Lambda$ of the objective function, we set $A_j=\sqrt{\lambda_j}V_j$ and $A_{\varphi+\ell}=W_\ell$, where $\lambda_j=\Lambda^{j/\varphi}$.
Rather than choosing their span at random, we constructed the $V_j$ so that the squared Frobenius norm of the projection of $X_0-X_\star$ onto their span equals $\varpi\|X_0-X_\star\|_{\F}^2$, where $\varpi\in(0,1)$ denotes the prescribed projection ratio.
This experiment was designed to verify conditioning and convergence behavior.

We present additional numerical results for matrix sensing with squared loss.
\cref{fig:matrix_sensing_align0p5,fig:matrix_sensing_align0p2} show the primal gap for $\varpi=0.5$ and $\varpi=0.2$, respectively.
Our methods outperformed the baseline methods in both cases.
In particular, RBNFW (\texttt{local3}) was the best-performing method in both cases.
APG obtained the smallest primal gap but required much more time than our methods.

\begin{figure}
    \centering
    \includegraphics[width=\textwidth]{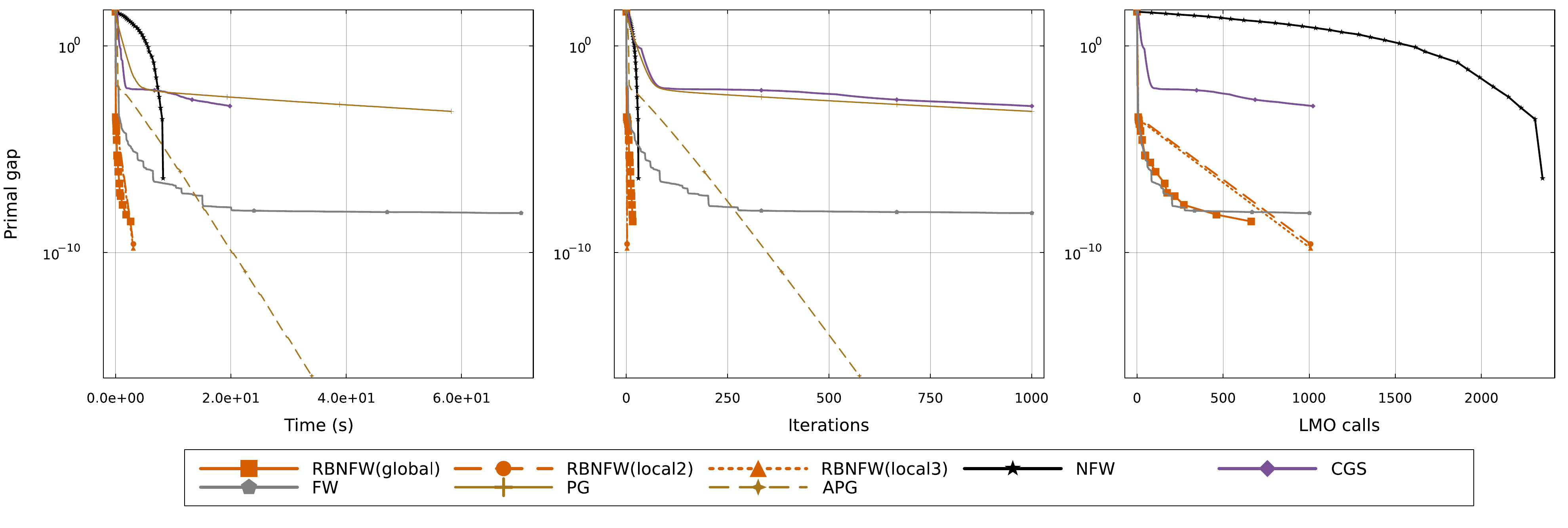}
    \caption{Primal gap for matrix sensing with $\varpi=0.5$.}
    \label{fig:matrix_sensing_align0p5}
    \smallskip
    \centering
    \includegraphics[width=\textwidth]{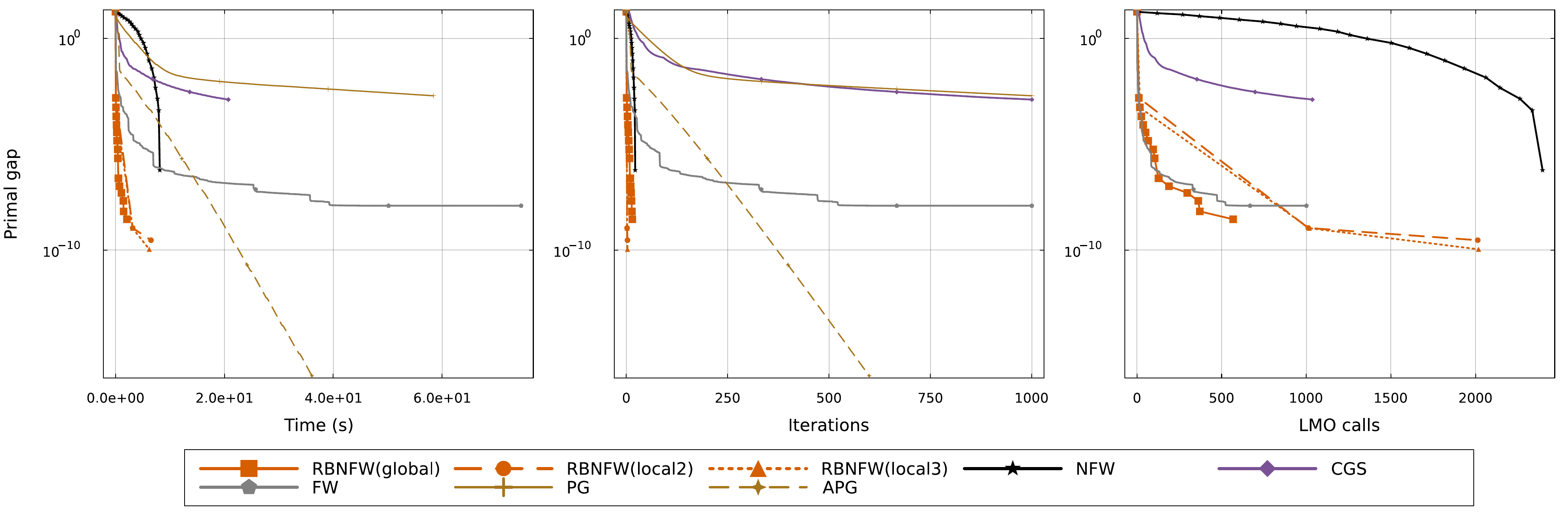}
    \caption{Primal gap for matrix sensing with $\varpi=0.2$.}
    \label{fig:matrix_sensing_align0p2}
\end{figure}

\subsection{Ridge-Regularized Logistic Regression}\label{appendix:ridge_logistic_regression}

We present additional numerical results for ridge-regularized logistic regression.
\cref{fig:logistic_a9a_l2_ball,fig:logistic_mushrooms_l2_ball,fig:logistic_mushrooms_polytope,fig:logistic_phishing_polytope,fig:logistic_w7a_polytope} show the primal gap for different datasets and constraint sets.
In \cref{fig:logistic_a9a_l2_ball,fig:logistic_mushrooms_l2_ball,fig:logistic_mushrooms_polytope}, the RBNFW variants outperformed the baseline methods.
In \cref{fig:logistic_phishing_polytope}, NFW also performed well.
In \cref{fig:logistic_w7a_polytope}, the performance of the RBNFW variants in terms of computation time was nearly identical to that of SOCGS.
Note that \cref{fig:logistic_mushrooms_polytope} does not include the results for away-step FW, and \cref{fig:logistic_phishing_polytope} does not include the results for CGS because they were very slow in these experiments.

\begin{figure}
    \centering
    \includegraphics[width=\textwidth]{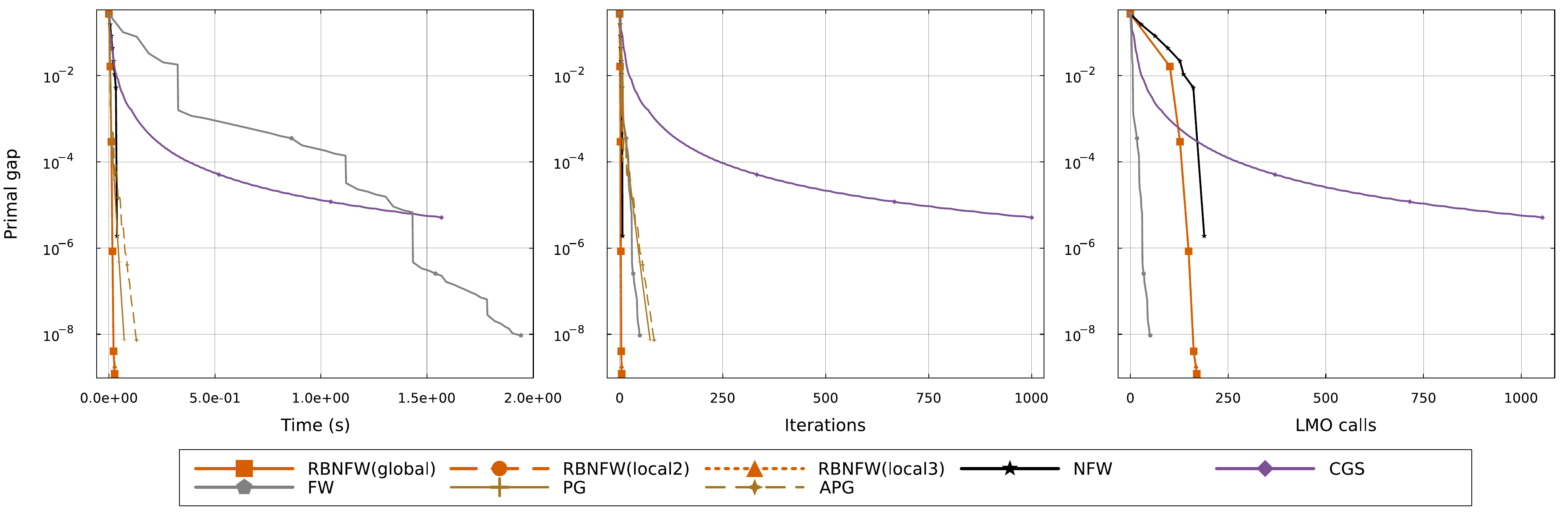}
    \caption{Primal gap for the a9a dataset with an $\ell_2$-ball constraint.}
    \label{fig:logistic_a9a_l2_ball}
    \smallskip
    \centering
    \includegraphics[width=\textwidth]{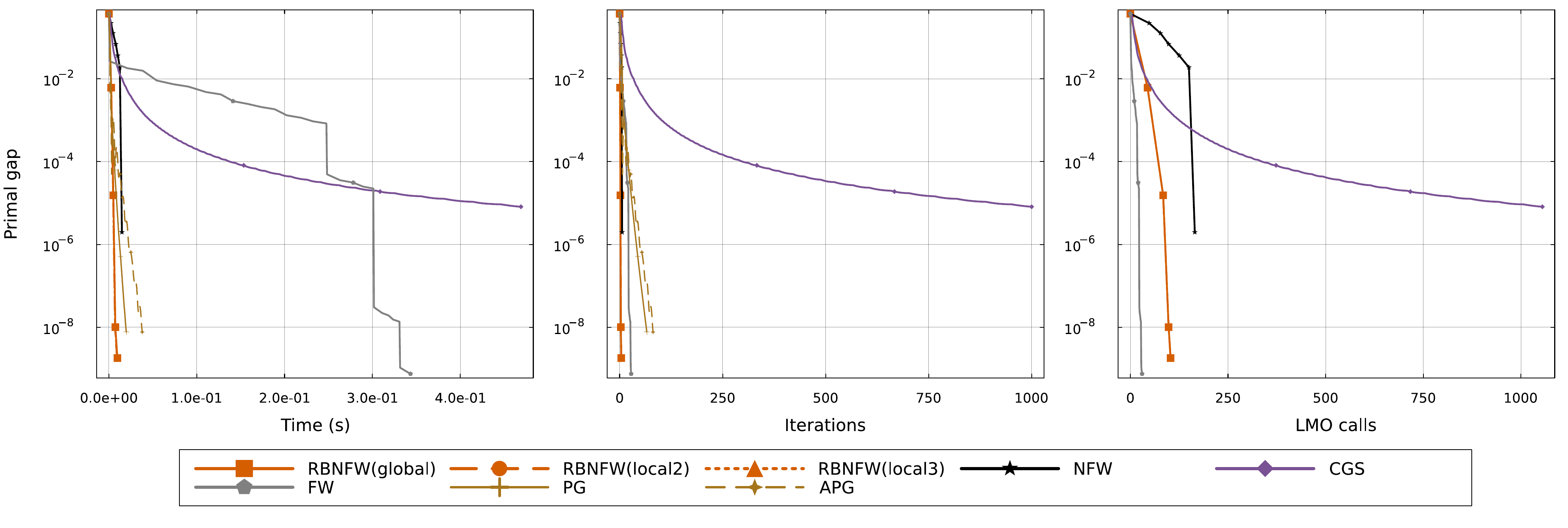}
    \caption{Primal gap for the mushrooms dataset with an $\ell_2$-ball constraint.}
    \label{fig:logistic_mushrooms_l2_ball}
\end{figure}

\begin{figure}
    \centering
    \includegraphics[width=\textwidth]{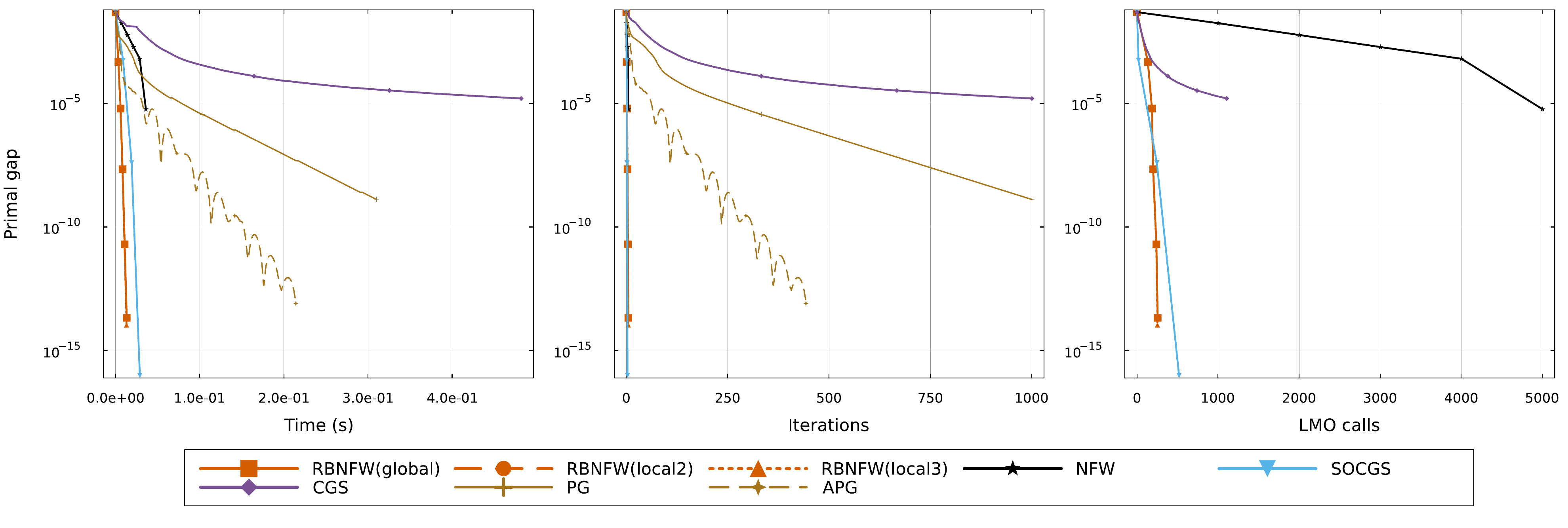}
    \caption{Primal gap for the mushrooms dataset with a sparse polytope constraint.}
    \label{fig:logistic_mushrooms_polytope}
    \smallskip
    \centering
    \includegraphics[width=\textwidth]{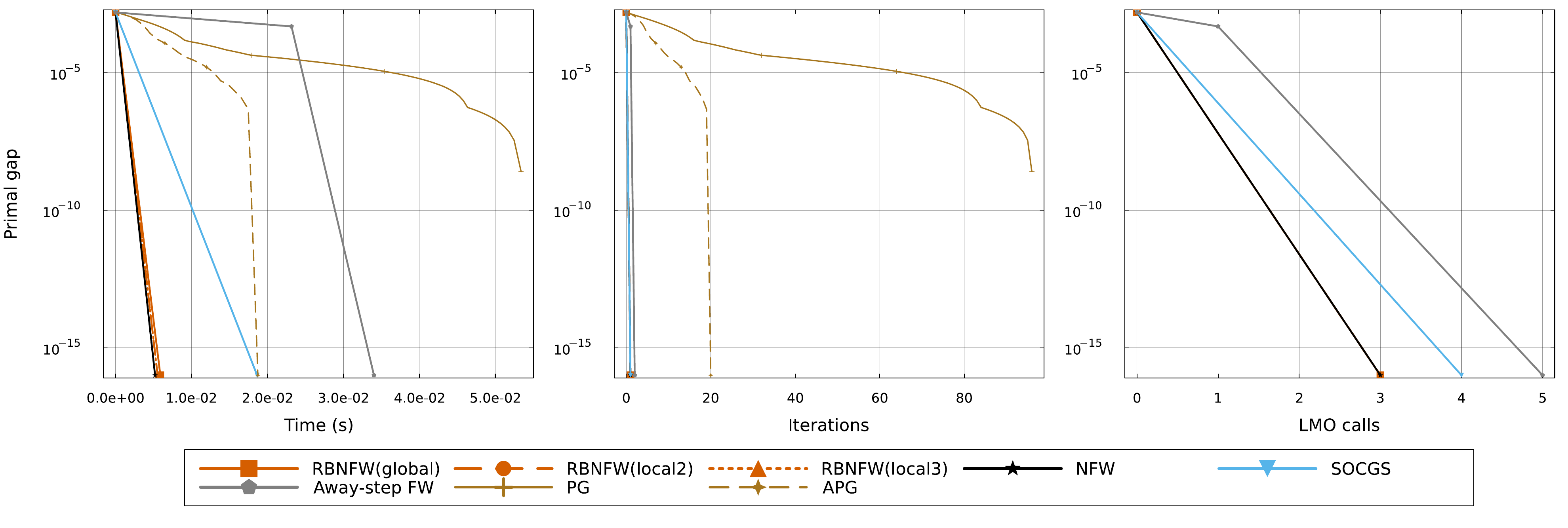}
    \caption{Primal gap for the phishing dataset with a sparse polytope constraint.}
    \label{fig:logistic_phishing_polytope}
    \smallskip
    \centering
    \includegraphics[width=\textwidth]{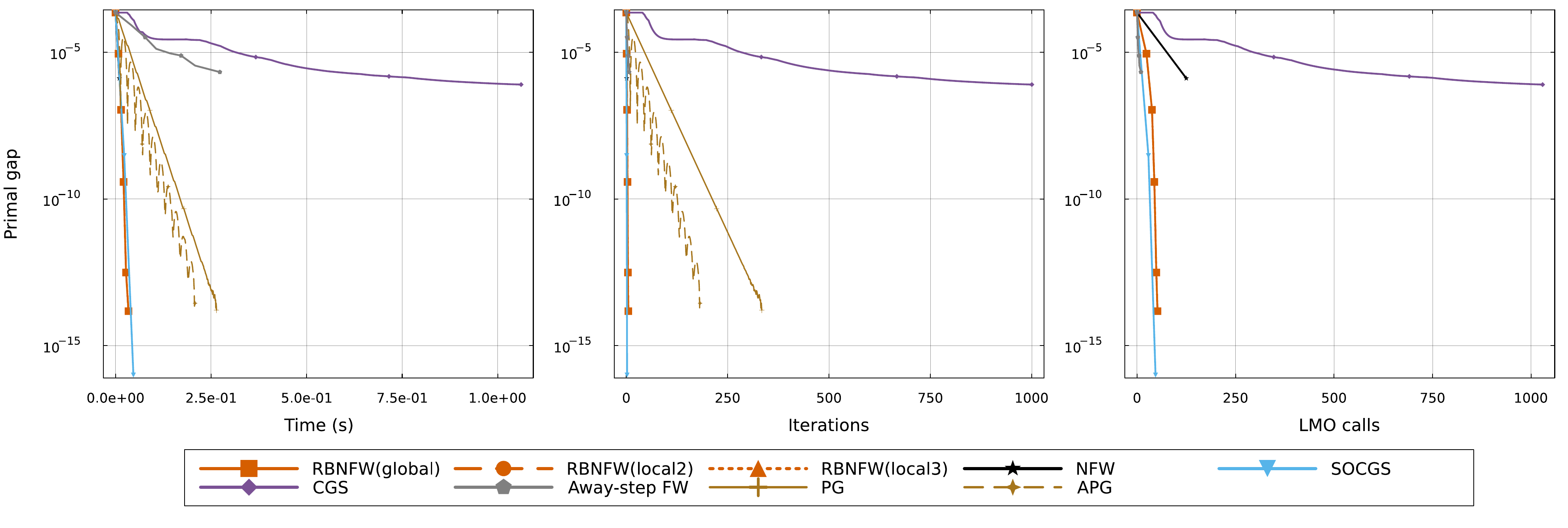}
    \caption{Primal gap for the w7a dataset with a sparse polytope constraint.}
    \label{fig:logistic_w7a_polytope}
\end{figure}

\subsection{Outer Iterations Reaching the Inner Iteration Limit}\label{appendix:inner_iterations_reach}
We capped inner solves at 1000 iterations, as in \citet[Section~4]{Carderera2026-pk}.
We report the affected outer iterations in \cref{tab:inner_iterations_reach}.
When the cap was reached before the prescribed accuracy was attained, we accepted the last inner iterate; this relaxation is not covered by our convergence analysis.
The prescribed accuracy depends on global problem constants and can be overly stringent in practice. 
RBNFW exhibited convergent behavior in the reported experiments and remained consistently stable, supporting its use in practice.

\begin{table}[!pht]
  \centering
  \caption{
    Number of outer iterations in which at least one inner solve reached
    the maximum of 1000 inner iterations.
    Multiple capped inner solves within the same outer iteration are counted
    only once.
    An entry ``--'' indicates that the corresponding method was not included
    in the corresponding experiment.
  }
  \label{tab:inner_iterations_reach}
  \resizebox{\textwidth}{!}{%
    \begin{tabular}{lrrrrrrrrrrr}
      \toprule
      & \multicolumn{3}{c}{Matrix sensing}
      & \multicolumn{3}{c}{Logistic: \(\ell_2\)-ball}
      & \multicolumn{5}{c}{Logistic: sparse polytope} \\
      \cmidrule(lr){2-4}
      \cmidrule(lr){5-7}
      \cmidrule(lr){8-12}
      Method
      & \(\varpi=0.2\)
      & \(\varpi=0.5\)
      & \(\varpi=0.8\)
      & covtype.binary
      & a9a
      & mushrooms
      & covtype.binary
      & a9a
      & mushrooms
      & phishing
      & w7a \\
      \midrule
      RBNFW (\texttt{global}) & 0 & 0 & 0 & 0 & 0 & 0 & 0 & 0 & 0 & 0 & 0 \\
      RBNFW (\texttt{local2}) & 2 & 1 & 0 & 0 & 0 & 0 & 0 & 0 & 0 & 0 & 0 \\
      RBNFW (\texttt{local3}) & 2 & 1 & 1 & 0 & 0 & 0 & 0 & 0 & 0 & 0 & 0 \\
      \midrule
      NFW            & 0 & 0 & 0 & 0 & 0 & 0 & 3 & 7 & 6 & 0 & 0 \\
      SOCGS          & -- & -- & -- & -- & -- & -- & 1 & 2 & 0 & 0 & 0 \\
      CGS            & 0 & 0 & 0 & 0 & 0 & 0 & 0 & 0 & 0 & -- & 0 \\
      \bottomrule
    \end{tabular}%
  }
\end{table}

\end{document}